\documentclass[12pt]{amsart}
\usepackage{amsmath,amsthm, amssymb}
\usepackage{graphicx,epic}
\usepackage{mathrsfs}
\usepackage{hyperref}
\usepackage{float,yhmath}
\newtheorem{theoremalph}{Theorem}

\newtheorem{theorem}{Theorem}[section]
\newtheorem{proposition}[theorem]{Proposition}
\newtheorem{lemma}[theorem]{Lemma}
\newtheorem{corollary}[theorem]{Corollary}

\newtheorem{definition}[theorem]{Definition}

\newtheorem{obs}{Observation}

\def\ie{{\em i.e.,\ }}
\def\AE{{\em a.e.\ }}
\def\eg{{\em e.g.\ }}
\def\eps{\varepsilon}
\def\N{{\mathbb N}}

\def\R{{\mathbb R}}
\def\C{{\mathbb C}}

\newcommand {\CA}{{\mathcal A}}
\newcommand {\CB}{{\mathcal B}}
\newcommand {\CC}{{\mathcal C}}

\newcommand {\CF}{{\mathcal F}}

\newcommand {\CH}{{\mathcal H}}
\newcommand {\CI}{{\mathcal I}}

\newcommand {\CK}{{\mathcal K}}
\newcommand {\CL}{{\mathcal L}}
\newcommand {\CM}{{\mathcal M}}

\newcommand {\CO}{{\mathcal O}}
\newcommand {\CP}{{\mathcal P}}

\newcommand {\CT}{{\mathcal T}}
\newcommand {\CU}{{\mathcal U}}
\newcommand {\CV}{{\mathcal V}}
\newcommand {\CW}{{\mathcal W}}
\newcommand {\CX}{{\mathcal X}}

\def\s{\sigma}
\def\l{\lambda}

\def\1{ {\hbox{{\it 1}} \!\! I} }

\def\al{\alpha}
\def\be{\beta}
\def\de{\delta}
\def\ga{\gamma}

\def\8{\infty}

\newcommand{\ol}{\overline}
\def\disp{\displaystyle}

\renewcommand{\S}{\Sigma}
\usepackage{float,color}
\newcommand{\wt}{\widetilde}
\newcommand{\wh}{\widehat}

\newcommand{\BBone}{{1\!\!1}}

\newcommand{\ninf}{{n\to+\8}}
\newcommand{\mv}{{\mathbf{v}}}
\newcommand{\mw}{{\mathbf{w}}}

\newcommand{\grosinte}[1]{\widering{#1}}
\newcommand{\inte}[1]{\mathring{#1}}

\renewcommand{\L}{\Lambda}

\def\io{\iota}
\def\I{[p_{-},p_{+}]}
\def\Seps{\S_{[-\eps,\eps]}}
\def\mus{\mu_{SRB}}
\def\ko{\kappa_{0}}
\def\k1{\kappa_{1}}

\def\TCT{\Theta_{\CT}}

\def\FS{F_{\S}}

\def\FM{F_{\CM}}

\def\vmp{{V_{\CM,\CP}}}
\def\vmo{{V_{\CM,0}}}
\def\tom{\tau_{\CM}}
\def\gm{g_{\CM}}
\def\rm{r_{\CM}}

\def\DR{{\Delta r}}

\def\Fi{\varphi}
\theoremstyle{definition}
\newtheorem{remark}{Remark}

\newtheorem*{nota}{Notation}

\newtheorem*{termino}{Terminology}

\begin{document}
\synctex =1
\title[Equilibrium states for SaPH-flows]
{Uniqueness of the measure of maximal entropy for singular hyperbolic flows in dimension 3 and more results on equilibrium states}
\author{Renaud Leplaideur }
\date{Version of \today}
\thanks{}

\thanks{}

\maketitle

\begin{abstract}
We prove that any 3-dimensional singular hyperbolic attractor admits  for any H\"older continuous potential $V$ at most one equilibrium state for $V$ among regular measures. We give a condition on $V$ which ensures that no singularity can be an equilibrium state. Thus, for these $V$'s, there exists a unique equilibrium state and it is a regular measure. Applying this for $V\equiv 0$, we show that any 3-dimensional singular hyperbolic attractor admits a unique measure of maximal entropy.
\end{abstract}

\section{Introduction}\label{sec:intro}

\subsection{Background}
This paper deals with Thermodynamic formalism for partially hyperbolic attractors with singularities in dimension 3. 
The Thermodynamic formalism has been introduced in Ergodic Theory  in the 70's by Ruelle, Sinai and Bowen (see \cite{BR75,Bowenlnm, Ruelle76, Sinai72}). Firstly studied for uniformly hyperbolic dynamical systems,  it has been a challenge  for many years, and still is, to extend it to non-uniformly hyperbolic dynamical systems. 

In this paper, we study existence and uniqueness of equilibrium states  for 3-dimensional partially hyperbolic attractors with singularities. The  main famous example in this class is the family of the Lorenz-like attractors. First introduced by Lorenz in \cite{Lorenz}, this class has several typical properties of chaotic dynamics: it is robust in the $C^{1}$-topology, every ergodic invariant measure is hyperbolic  but  the attractors themselves are not  hyperbolic. 
For this class we prove  in Theorem \ref{th-main1} uniqueness of the relative equilibrium state for any H\"older continuous potential among non-singular measures. Theorem \ref{th-main2} gives a large class of H\"older continuous potentials for which there is a unique equilibrium state and it is a regular measure.
Theorem \ref{th-main3} states uniqueness of the measure of maximal entropy. 

Even if a large variety of potentials is possible, it is sometimes considered that two of them are the most important. The nul-function, because it furnishes the measure with maximal entropy and the logarithm of the unstable Jacobian which gives the SRB-measure (sometimes called $u$-Gibbs state). Beyond the interest of these potential, it is also noteworthy that they are also the easiest to study/construct. For the first one, because there is no problem to control distorsions for the nul-function. For the second one,  because the geometrical properties of the $u$-Gibbs states immediately gives for free the ``conformal-measure''. 

 Most of the results for the singular hyperbolic flows deal with their dynamical properties. They can be classified in several kinds of result. The ones which deal with $C^{1}$-generic properties, robustness  and homoclinic classes (see \eg \cite{ABCD07,ABC11, BCW09, BC04}). Other results study mixing properties (see \eg \cite{AraujoMelbourne19, AraujoMelbourne17, AraujoMelbourne16, AraujoMelbourneVarandas15}). At last, results for general properties, including existence and uniqueness of the SRB measure or generalizations as the Rovella Attractor, (see \eg \cite{APPV09, MetzgerMorales08, MetzgerMorales06, MPP04,MPP99}). 

About Ergodic results, \ie existence and uniqueness of special invariant measures, all the known results deal with the SRB measure. We remind that  for a map $f:X\to X$ and a potential $V:X\to \R$, an equilibrium state is an invariant (probability) measure which maximizes the free energy of the potential $V$:
$$h_{\mu}+\int V\,d\mu=\max\left\{h_{\nu}+\int V\,d\nu\right\}.$$
For flows, we consider the time-1 map. 

The SRB measure is usually obtained as a $u$-Gibbs state, that is for a special case\footnote{ More precisely, $V=-log J^{cu}$ or $V=-\log J^{u}$ depending of the assumption on the non-uniformly hyperbolic system. } of $V$.
We mention the general result in \cite{Qiu}. For singular hyperbolic attractors in dimension 3 existence and uniqueness of the SRB mesure is done in \cite{APPV09}. In higher dimension it is done in \cite{Lep-Yang}. 
As far as we know, very few results exist for the measure of maximal entropy. For diffeomorphisms, we mention the generic result \cite{BCS}.
Hence,  and still as far as we know, our result here is thus the first result dealing with equilibrium state  for general potentials for these systems and also for the special case of maximal entropy. 

For any continuous potential, the existence of equilibrium states comes from the upper semi-continuity of entropy. We refer to \cite{Bowenlnm} for classical results on equilibrium states. Upper semi continuity of the entropy follows from expansiveness and expansiveness for singular hyperbolic attractors in dimension 3 is proved in \cite{APPV09}. Therefore, the existence of equilibrium state for any continuous potential was already known. Hence, the novelty in this paper is uniqueness of these measures, the fact that they have full support and that they are local equilibrium states with some local Gibbs property.

\subsection{Settings and statement of  results}

Let $X$ be a vector field on a $d$-dimensional manifold $M$, and $f_t$ be the flow generated by $X$. We recall that a compact invariant set $\Lambda$ is called a topological attractor if
\begin{itemize}

\item there is an open neighborhood $U$ of $\Lambda$ such that $\cap_{t\ge 0}f_t({\overline U})=\Lambda$,

\item $\Lambda$ is transitive, i.e., there is a point $x\in\Lambda$ with a dense forward-orbit. 

\end{itemize}

A compact invariant set $\Lambda$ is called a \emph{singular hyperbolic attractor} (see  \cite{MetzgerMorales08,ZGW08}) if it is a topological attractor, with at least one \emph{singularity} $\sigma$, which means that $\sigma$ satisfies  $X(\sigma)=0$.
Moreover, there is a continuous invariant splitting $T_\Lambda M=E^{ss}\oplus E^{cu}$ of $Df_t$ together with constants $C>0$ and $\lambda>0$ such that
\begin{itemize}


\item Domination: for any $x\in\Lambda$ and any $t>0$, $\|{D}f_t|_{E^{ss}(x)}\|\|{Df_{-t}}|_{E^{cu}(f_t(x))}\|\le C{e}^{-\lambda t}$.

\item Contraction: for any $x\in\Lambda$ and any $t>0$, $\|{D}f_t|_{E^{ss}(x)}\|\le C{e}^{-\lambda t}$.

\item Sectional expansion: for any $x$, {$E^{cu}(x)$ contains two non-collinear vectors}, and any $t>0$, for every pair of non-coltinear vectors $\mv$ and $\mw$ in $E^{cu}(x)$, 
$|\det {D}{f_{t}}|_{{span}<\mv,\mw>}|\ge C{e}^{\lambda{t}}.$

\end{itemize}

We  emphasize that one of the difficulties to study these attractors is that  the singularity may belong  to the attractor  and may be accumulated by recurrent regular orbits. {Since the uniformly hyperbolic case has already been well-understood since \cite{BR75}, we assume that the attractor does contain at least one singularity.}

The model that we have in mind clearly is the Lorenz attractor, and our construction needs some regularity for the strong stable foliation. In a very recent work, \cite{AraujoMelbourne19}, it is proved that the strong stable foliation is H\"older continuous. It has also been  proved in \cite{AraujoMelbourne16} that the foliation is even Lipschitz-continuous  for the Lorenz attractor and close attractors.

\bigskip
We recall that entropy for a flow is defined as being the entropy of the time-1 map $f_{1}$.

\begin{definition}
\label{def-equilstate}
For $V:\L\to\R$ H\"older continuous called a potential, an equilibrium state is a $f_{1}$-invariant probability measure $\mu$ which maximizes the free energy for $V$, that is 
$$h_{\mu}+\int V\,d\mu=\max\left\{h_{\nu}+\int V\,d\nu\right\}.$$
This maximum is called the pressure for $V$ and is denoted by $\CP(V)$. 
\end{definition} 
 We remind that entropy is affine: $\disp h_{t\mu_{1}+(1-t)\mu_{2}}=th_{\mu_{1}}+(1-t)h_{\mu_{2}}$. This yields that any equilibrium state is a convex combination of ergodic equilibrium states. Consequently, all equilibrium states are well known as soon as ergodic equilibrium states are all known.

\paragraph{\bf Some vocabulary}

We say that $x\in\L$ is {\it regular} if it is not a singularity. 
We say that an ergodic measure is {\it regular} if no singularity has positive $\mu$-measure. 
A point $x$ is said to be {\it regular with respect to some invariant ergodic measure}, say $\mu$, if any property true $\mu$-almost everywhere holds for $x$. We shall also say $x$ in $\mu$-regular.

Now we have the dichotomy: if $\mu$ is an ergodic equilibrium state for $V$ then 
\begin{description}
\item[C1] either $\mu=\delta_{\s}$ for some singularity $\s$,
\item[C2] or $\mu$ is regular. 
\end{description}

With all these settings we prove in this paper:

\begin{theoremalph}\label{th-main1}
Assume that $\Lambda$ is a singular hyperbolic attractor of a $C^2$ 3-dimensional vector field $X$. For every H\"older $V:\Lambda\to\R$  there exists at most one unique regular equilibrium state for $V$. If it does, then it has full support.

More precisely, only two situations may happen for $\be\to \CP(\be.V)$ and $\be>0$. 
\begin{enumerate}
\item Either it is strictly convex and analytic for any $\be>0$, and then there exists a unique equilibrium state (for every $\be>0$) and it is a regular measure with full support. 
\item Or there exists $\be_{c}>0$ such that the previous case holds for every $0\le \be<\be_{c}$, and for every $\be>\be_{c}$ all the equilibrium states are supported on singularities. 
\end{enumerate}

\end{theoremalph}

We remind that a $V$-maximizing measure is a measure which gives maximal value for the integral of $V$ among all invariant measures. 
 
%

\begin{theoremalph}
\label{th-main2}
If $V$ is such that no measure supported on a singularity is a $V$-maximizing measure, then for any $\be\ge 0$, there exists a unique equilibrium state for $\be.V$. Moreover, the pressure function $\CP(\be.V)$ is analytic.  
\end{theoremalph}

Finally for the case $V\equiv 0$ we have:

\begin{theoremalph}
\label{th-main3}
There exists a unique measure with maximal entropy. 
\end{theoremalph}

\subsection{Plan of the paper and main ingredients of the proof}
The first main ingredient in the proof is to construct a topological object called a mille-feuilles. Roughly speaking, a mille-feuilles is a cross-section to the flow direction with several properties: 
\begin{enumerate}
\item It is a compact collection of \emph{pseudo unstable curves} and is fibered by strong stable curves.
\item It has the rectangle property as in \cite{Bowenlnm}.
\item The first return by  the flow has the Markov property. 
\end{enumerate}
{\it Pseudo unstable curves} are in spirit candidates to be the traces in the cross section of true central-unstable leaves. We remind that these leaves do not exists everywhere but at least every regular point for any regular measure does have such a leaf. One of the difficulty is to define them without  any reference to any pre-chosen invariant measure in view to define the mille-feuille as a topological object.

For the well-known classical Lorenz attractor, it is known that there is a global transversal section. This is the object we want to mimic in our construction of a good cross section. This is done in Section \ref{sec-cross}. We mention that in \cite{AraujoMelbourne19} one of the work consists in constructing a global cross section. We emphasize that in our case the construction in only local. 

The construction of the mille-feuilles is done in Section \ref{sec-millefeuilles}. We first construct a pre-mille-feuilles and then the true mille-feuilles (and also a generalized mille-feuilles). The motivation to construct this topological object is the following.

Still in the Lorenz attractor, the return map in the global Poincaré section has a natural skew-product structure over the interval. There is thus the 3d-dynamics from the flow, the 2d-dynamics in the section and the 1d-dynamics in the interval. 
This has been a motivation to study piecewise expanding maps on the interval with singularities. In that direction we mention works of Hofbauer (see \cite{Hofbauer-inter}) and Buzzi (see \eg \cite{buzzi-survey}). 

To study these 1d-systems, in particular to get existence and uniqueness of the measure of maximal entropy, it is noteworthy that one method consists in studying the Hofbauer-diagram which is a kind of subshift extension over the 1d-dynamics. 

In other words, starting from a 1d-dynamics, the Hofbauer diagram is a kind of abstract 2d-dynamics over the 1d map. 
We believe that this passage to the 2d dynamics is one of the key point to study Thermodynamics formalism for theses maps. In our mind the mille-feuilles we construct below is a geometrical representation of the Hofbauer-diagram. As it is constructed via geometrical decriptions, we believe it makes the 2d-dynamics less abstract and thus easier to be understood. 

Section \ref{sec-inducingschemealanono} re-emploies the technic of local equilibrium state developed by the author along years and summarized  in \cite{lep-survey}. We tried to make it as self-contained as possible. We remind that the key point is to define an induced scheme and the notion of local equilibrium state. We deeply use the Abramov formula which makes a link between the entropy of an induced map and the original one.  
Roughly speaking we show that the mille-feuilles constructed previously admits a unique local equilbrium state for a good induced potential and that this measure can be opened-out in a global invariant measure which turns out to be the natural candidate to be a regular global equilibrium state.

Section \ref{sec-endproof} is devoted to the end of the proofs of the Theorems. The proof of  the first part of Theorem \ref{th-main1} simply consists in separating the cases. Either there is no regular equilibrium state, or there is one, and then it must be unique because it must coincide with the unique local equilibrium on any mille-feuilles. 
The two other proofs are then simple consequences of the fact that assumptions yield that the second case holds and not the first one. 

Finally, analyticity for $\CP(\be)$ is proved in the last subsection.

\subsection{Acknowledgment}
Part of this work has been written as the author was visiting D. Yang at Soochow University. We would like to thank Soochow University for kind hospitality and D. Yang for having answering to our technical questions. 

The notion of mille-feuilles and some of the ideas behind were already present in a unpublished paper of the author with V. Pinheiro that can be find here \cite{lep-Pinheiro}.

\section{Cross-sections with good properties}\label{sec-cross}

\subsection{Special local cross-sections}

\begin{definition}
\label{def-lambdacunstableman}
Let $\wt\l$ be  positive real number. Let $x$ be in $\L$. 
We say that $x$ is $\wt\l$-hyperbolic if there exists  $\kappa_{0}>0$ such that $$W^{uu}_{loc}(x):=\left\{y,\ \forall\,t>0\ d(f_{-t}(y),f_{-t}(x))\le \ko e^{-\wt\l t}d(x,y)\right\}$$
is a non-trivial  immersed $C^{1}$-manifold. 
\end{definition}
if $x$ is $\wt\l$-hyperbolic, then so is any $f_{t}(x)$, with $t\in \R$. Moreover, we claim that 
$$\limsup_{t\to-\8}\frac1{|t|}\log |Df_{t}(x)|_{E^{uu}(x)}|\le -\wt\l,$$
holds where $E^{uu}(x)=T_{x}W^{uu}_{loc}(x)$. 
This yields a splitting $$E^{cu}(f_{t}(x))=E^{uu}(f_{t}(x))\oplus <X>.$$

Furthermore, if $0<\wt\l'<\wt\l$, every $\wt\l$-hyperbolic point is also $\wt\l'$-hyperbolic. 

If $x$ is $\wt\l$-hyperbolic, then the set $\disp\S:=\bigcup_{y\in W^{uu}_{loc}(x)}W^{ss}_{loc}(y)$ is transverse to the flow direction (at least locally around $x$). 
It is  a continuous surface and we remind that the stable foliation is H\"older continuous. 



\begin{definition}
\label{def-transversal}
A cross-section is a set satisfying the following:
\begin{enumerate}
\item It is a local proper topological surface $\S=\ol{\inte\S}$ (for the 2d-topology),
\item it is $W^{ss}_{loc }$-foliated,
\item it is transversal to the vector field $X$,
\item it contains one $\ko$-unstable curve $W^{uu}_{loc}(x_{0})$ (for some $\ko$), referred to as the basis of the cross-section. 
\item Extremities of the basis do not belong to stable foliation of some periodic point.
\end{enumerate}
If $\S$ is a cross-section, the $\eps$-time neighborhood of $\S$ is the set $\disp\bigcup_{t\in[-\eps,\eps]}f_{t}(\S)$. It will be denoted by $\S_{[-\eps,\eps]}$. 
\end{definition}

For $x$ in the basis  $W^{uu}_{loc}(x_{0})$ of $\S$, $T_{x}W^{uu}_{loc}(x_{0})$ and $T_{x}W^{ss}_{loc}$ are well defined and are non-colinear. We set 
$$T_{x}\S:=\text{span}\{T_{x}W^{uu}_{loc}(x_{0}), T_{x}W^{ss}_{loc}(x_{0})\}. $$

For simplicity we shall always assume that a cross-section satisfies  for every $x$ of the basis
$$\measuredangle(T_{x}\S,X(x))>\frac\pi4.$$
By definition a cross-section does not contain singularity.

\begin{termino}
We say that $\S$ has positive* $\mu$-measure if $\disp\mu(\Seps)>0$ holds for any positive $\eps$.  
\end{termino}
\begin{remark}
\label{remasterix}
In general we shall put some asterisque * do indicate the property relative to measures has to be understood following our terminology. 
$\blacksquare$\end{remark}

\begin{definition}
\label{def-WssSigma}
If $\S$ is a cross section and $x$ belongs to $\L\cap\S$, the unstable local leaf $W_{\S}^{ss}(x)$ is defined by 
$$W_{\S}^{ss}(x)=\S\cap W^{ss}_{loc}(x).$$
\end{definition}

Let $\CT$ be the basis of $\S$ and $\Theta_{\CT}$ be the natural projection onto $\CT$: 
$$\Theta_{\CT}(x)=\Theta_{\CT}(y)\iff y\in W^{ss}_{\S}(x).$$


%
%

\bigskip

If $\S$ is a cross-section one can define the first return map, say $F_{\S}$, into $\S$ by the flow $f_{t}$. The return time is denoted by $r_{\S}$:
$$F_{\S}(x)=f_{r_{\S}(x)}(x),\text{ with }r_{\S}(x)=\min\{t>0,\ f_{t}(x)\in\S\}\le +\8.$$

\begin{lemma}
\label{lem-returncrosssec}
For any cross section, there are points with finite return time. 
\end{lemma}
\begin{proof}
For a cross-section $\S$, consider the compact set with non-empty interior, $\disp\Seps$. Transitivity for $f_{t}$ yields existence of a dense set of point in $\disp\Seps$ having dense orbit. All these points have infinitely many returns into $\disp\Seps$ (with return time $>\eps$ ). If $y$ is such a point and $r$ is the first return, then, up to a time-translation one may assume $y\in \S$ and  $f_{r'}(y)\in\S$ holds for some $r-\eps<r'<r$. 
\end{proof}

\begin{lemma}
\label{lem-perioddense}
For a cross-section $\S$, $F_{\S}$-periodic orbits are dense in $\S$.
\end{lemma}
\begin{proof}
Actually periodic points are dense in $\L$.  This is a direct consequence of the existence of  $\mus$. This measure has full support and \AE point is regular.  
\end{proof}

Let $\S$ be a cross-section such that  $\Seps$ has positive measure for some regular measure $\mu$. Then, the Main Theorem on Special Representation of Flows (see \cite{CFS}) yields that one can locally represent the conditional measure $\mu$ as $d\mu\propto d\mu^{\S}\otimes dt$. Moreover, $\mu^{\S}$ is $F_{\S}$-invariant. 

\begin{definition}
\label{def-projsurS}
Let $\S$ be a cross-section. An $\eps$-identification $\iota$ is the map from $\disp\Seps$ onto $\S$ defined by 
$$\io(f_{t}(x))=x.$$ 
It is defined for some small positive $\eps$. 
\end{definition}

\begin{remark}
\label{rem-iotabilip}
We point-out that if $\S'$ is another cross-section defined in the neighborhood $\disp\Seps$, then the restriction of $\io$  from $\S'$ to $\S$ is bi-Lipschitz with  constants  only depending on the slopes of the cross-sections and on the size of the neighborhood. 
$\blacksquare$\end{remark}

\subsection{S-adapted cross sections and associated 1-dimensional dynamics}

\subsubsection{Definition and abundance of SACS}

\begin{definition}
\label{def-goocrosssec}
An $s$-adapted cross section (SACS in short)  is a cross-section $\S$ such that for every $x\in\L\cap\S$  
$$F_{\S}(W^{ss}_{\S}(x))\subset \grosinte{W^{ss}_{\S}(F_{\S}(x))}$$
holds as soon as $F_{\S}(x)$ is well-defined.
\end{definition}

\begin{proposition}
\label{prop-goodcrosssec}
Any point in $\Gamma_{\ko}$ for some $\ko$ belongs to the interior of some SACS. 
\end{proposition}

The proof of Proposition \ref{prop-goodcrosssec} needs a lemma which gives a topological property for stable leaves. 

\begin{lemma}[see \cite{AAPP} Lem. 3.2]
\label{lemma-nointervalinwss}
Let $x$ be in $\L$. Let $W^{ss}_{loc}(x)$ be a piece of local stable manifold. Then, $\L\cap W^{ss}_{loc}(x)$ is totally disconnected. 
\end{lemma}
\begin{proof}
The proof is done by contradiction. Let us assume that the small $s$-interval $W^{ss}_{loc}(x)$ is included in $\L$. Then we consider the map $f_{-1}$. It is partially hyperbolic with uniformly expanding bundle. Using \cite{BV00}, it admits an invariant measure with absolutely continuous disintegration along unstable leaf. 

Going back to the initial dynamics, $f_{1}$ admits an invariant measure with absolutely continuous disintegration on stable leaves. Consider then a small piece of stable leaf, say $W^{ss}_{loc}(y)$ where $Leb^{ss}$-almost every point is in $\L$. All these points $z$ admit a  local strong-unstable manifold with positive lenght. By construction $z$ is in $\L$ and as $\L$ is an attractor, any forward image  of $W^{uu}_{loc}(z)$  stays in the neighborhood of $\L$. By construction, any backward image of $W^{uu}_{loc}(z)$ also stays in the neighborhood of $\L$. This yields that $\L$ has positive Lebesgue measure. This is in contradiction within $\L$ being an attractor  (see \cite{AAPP}). 

\end{proof}

\begin{proof}[Proof of Prop. \ref{prop-goodcrosssec}]
Let $x\in\L$  be a regular point with respect to a regular measure $\mu$. Consider a small stable leaf $W^{ss}_{\de}(x)$ and adjust the length such that both extremal points are not in $\L$. Lemma \ref{lemma-nointervalinwss} shows this is possible.

Call these two extremities $x_{-}$ and $x_{+}$. There exists small balls, say $B(x_{\pm},\eps)$ centered in $x_{\pm}$ with empty intersection with $\L$. 
Thus, there exist positive numbers $\de_{\pm}$ such that the $s$-intervals $[x_{-},x_{-}+\delta_{-}]$ and $[x_{+}-\de_{+},x_{+}]$ have empty intersection with $\L$.
The picture (see Fig. \ref{Fig-halt1}) looks like a dumbbell with bar equal to $W^{ss}_{\de}(x)$ and the two balls at extremities (outside $\L$). 

\begin{figure}[htbp]
\begin{center}
\includegraphics[scale=0.5]{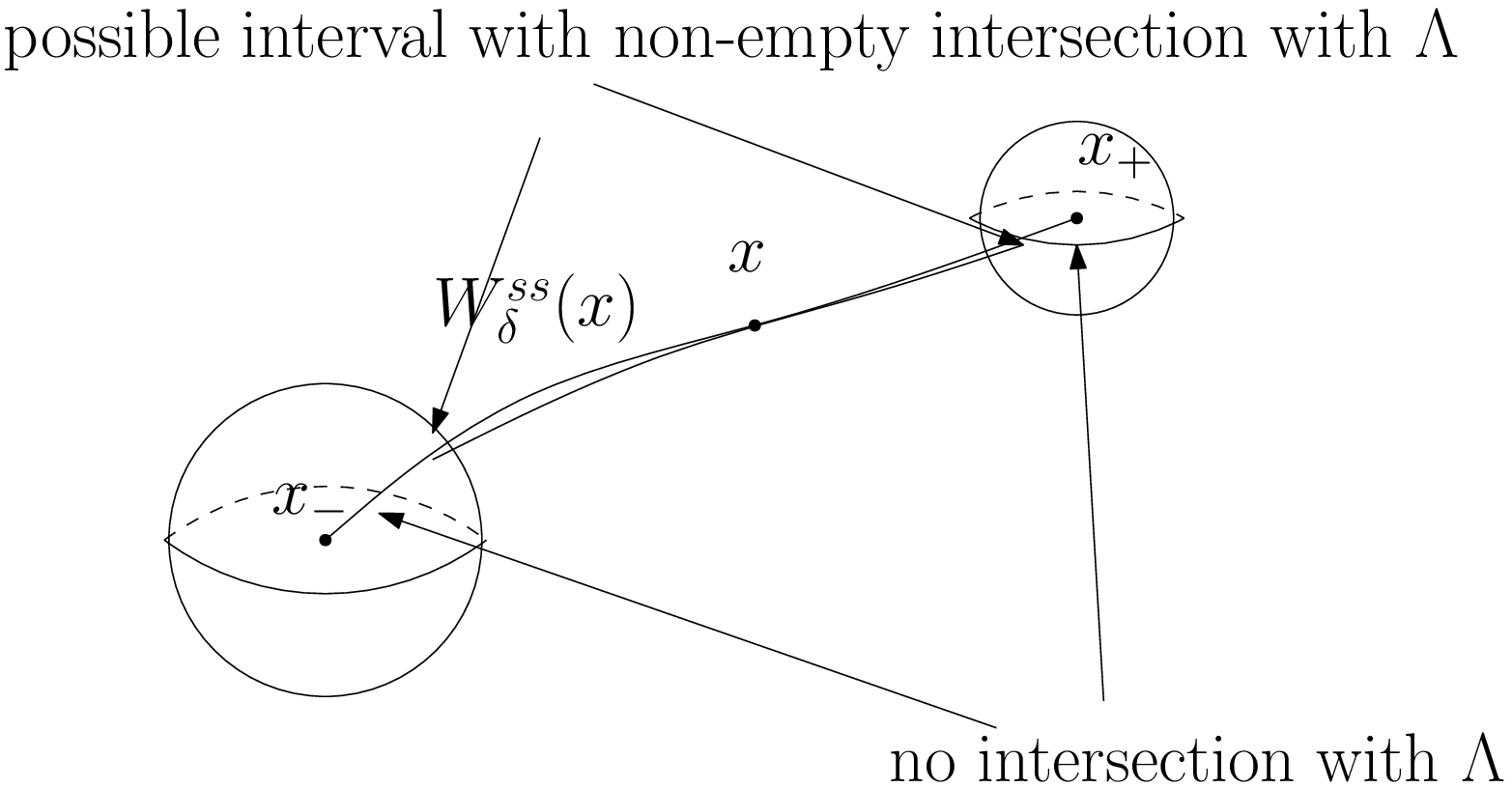}
\caption{$W^{ss}_{\de}(x)$ with extremities outside $\L$.}
\label{Fig-halt1}
\end{center}
\end{figure}

We remind that there exists an open neighborhood $B(x,\eps)$ of $x$ such that all the points in that neighborhood have a local stable manifold. This holds because $\Lambda$ is an attractor. 
Then, consider $W^{uu}_{loc}(x)\cap B(x,\eps)$, and construct the surface $\S$ obtained by taking the union of all of $W^{ss}_{\de}(y)$ where $y$ runs over $W^{uu}_{loc}(x)\cap B(x,\eps)$ and $W^{ss}_{\de}(y)$ is adjusted such that extremities are in $B(x_{\pm},\eps)$.  We also adjust the size of $W^{uu}_{loc}(x)$ to be sure that extremal points do not belong to the stable manifold of some periodic point. Then, we claim that $\S$ is an S-adapted cross section. 

\begin{figure}[htbp]
\begin{center}
\includegraphics[scale=0.6]{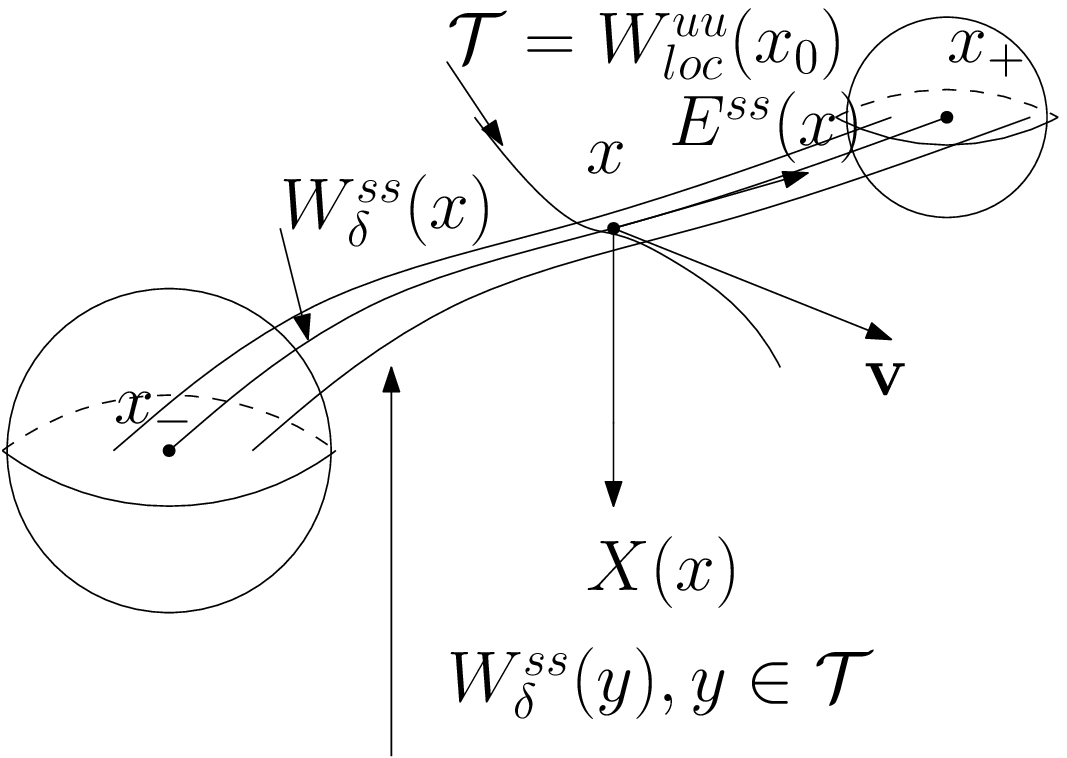}
\caption{Construction of $\S$}
\label{Fig-halt3}
\end{center}
\end{figure}

Indeed,  if $x$ is not a periodic point, one can increase the first return time in $\S$ by choosing sufficiently small $\eps$. Hence, one choose such an $\eps$ such that for any return time $n$ (of any point) $2\de. \l^{n}<<\min(\de_{-},\de_{+})$. 
 Then, by construction, any return in $\S$ maps a piece of unstable leaf $W^{ss}_{\S}$ on a smaller piece with length much smaller than $\de_{\pm}$ (see Fig. \ref{fig-halt4}).  
 
 \begin{figure}[htbp]
\begin{center}
\includegraphics[scale=.5]{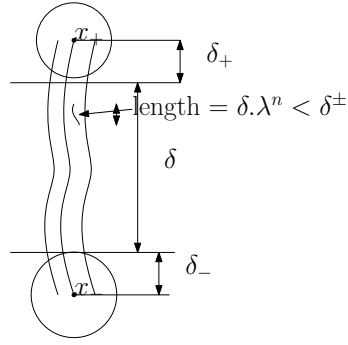}
\caption{Return too small to cross the balls of the dumbbell}
\label{fig-halt4}
\end{center}
\end{figure}

If $x$ is periodic, one can assume it is fixed. Then no other point in some neighborhood is fixed. Again, choosing a very small $\eps$ one can increase any return time (except the one for $W^{ss}_{\S}(x)$) and the same argument holds. 
\end{proof}

The proof immediately extends to a more general result
\begin{proposition}
\label{prop2-goodcrosssec}
For any regular measure $\mu$, for any  regular point $x$ with respect to $\mu$, there exists $\ko$ such that
$x$ belongs to the interior of some SACS.
\end{proposition}

\subsubsection{Dynamics in a S-adapted cross section}
From now on, one considers an $S$-adapted cross section $\S$. The first return map $F_{\S}$ has been defined above. 

 By definition,  $F_{\S}(W^{ss}_{\S}(x))\subset W^{ss}_{\S}(F_{\S}(x))$  holds because $\S$ is s-adapted. 
 Moreover, there is a canonical equivalent relation on $\S$, 
$x\sim y\iff y\in W^{ss}_{\S}(x)$. We recall that $\CT$ denotes the basis of the cross-section and $\Theta_{\CT}$ is the canonical projection onto the basis. 

As the cross section is $s$-adapted equality
$$\Theta_{\CT}\circ F_{\S}=\Theta_{\CT}\circ F_{\S}\circ\Theta_{\CT}$$
holds and defines a map $g_{\CT}:\CT\to\CT$ by $g_{\CT}:=\Theta_{\CT}\circ F_{\S}$. 

\begin{figure}[htbp]
\begin{center}
\includegraphics[scale=0.4]{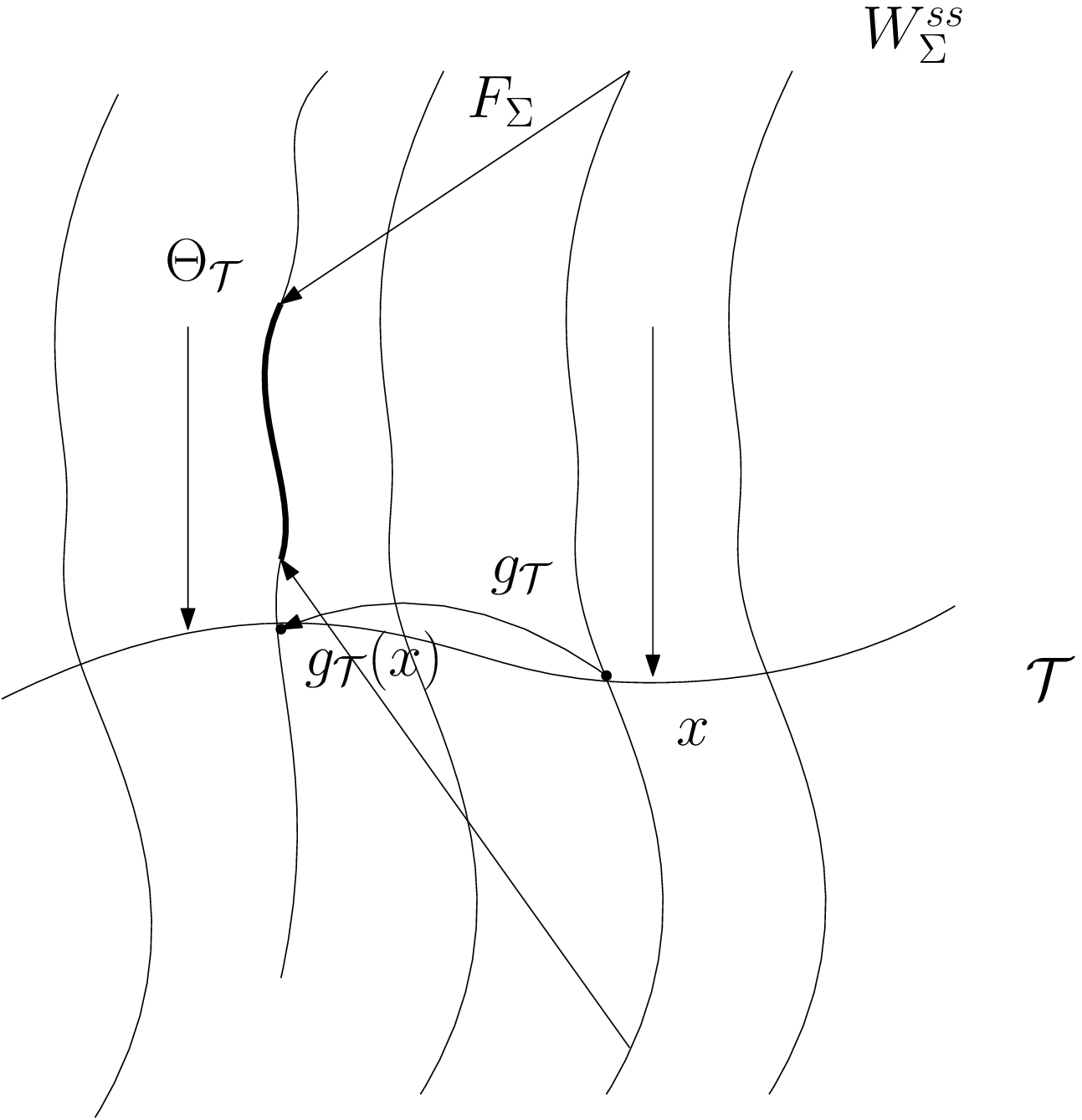}
\caption{Induced dynamics with skew product structure}
\label{Fif-mapsFSig}
\end{center}
\end{figure}

%

The next lemma states a kind of expansion in the trace of the unstable direction in $\S$. 

\begin{lemma}
\label{lem-expandeuus} There exists $C'$ such that for every $x$ in $\L\cap \S$ for which $F_{\S}(x)=f_{r_{\S}(x)}(x)$ is well-defined, 

\begin{equation}
\label{equ-expandeuus}
|Df_{r_{\S}(x)}(x)_{|X(x)^{\bot}}\wedge X(F_{\S}(x))|\ge C'.e^{r_{\S}(x).\l}
\end{equation}
holds, where $X(x)^{\bot}$ stands for the orthogonal direction to $X(x)$ in $E^{cu}(x)$. 
\end{lemma}

 \begin{proof}
$Df_{r_{\S}(x)}(x)$ maps $X(x)$ on $X(F_{\S}(x))$. Because the cross section $\S$ is small $X(x)\sim X(F_{\S})(x)$, both in direction and in norm. 

On the other hand $Df_{r_{\S}(x)}(x)$ expands by a factor larger than $C.e^{r_{\S}(x).\l}$  the  surface generated by $X(x)$ and $X(x)^{\bot}$. 
\end{proof}

%
%
%
%
%
\subsection{$(\S,u)$-curves}
The goal of this subsection is to determine the shape of the intersection of $\Lambda$ with a SACS $\S$.

\begin{definition}
\label{def-sucurveconti}
A $cu$-surface is a surface tangent to $E^{cu}$. 
A $(\S,u)$-curve in $\S$ is the intersection of a $cu$-surface with $\S$ satisfying:
\begin{enumerate}
\item it forms a connected graph $W^{cu}_{\S,loc}$ over an interval in the basis $\CT$,
\item for all $n\ge 0$, for all $x$ in $W^{cu}_{\S,loc}$, $F^{-n}_{\S}(x)$ is well defined,
\item for all $n\ge 0$, $F^{-n}_{\S}(W^{cu}_{\S,loc})$ is a connected graph over some interval of $\CT$. 
\end{enumerate}
\end{definition}

Next proposition show that $(\S,u)$-curves do exist. Actually, we expect  (see Proposition \ref{prop-descrilambda}) that any $(\S,u)$-curve is as described in the proposition. 

\begin{proposition}
\label{prop-sucurvesconti}
Let $\mu$ be a regular invariant measure such that $\Seps$ has positive $\mu$-measure. 
Let $x\in \S$ be $\mu$-regular. 

Then for any sufficiently small $\rho$, $\cup_{t\in[-\eps,\eps]}f_{t}(W^{uu}_{\rho}(x))\cap\S$ is a $(\S,u)$-curve which contains $x$ (see Fig. \ref{Fig-galeo}). 
\end{proposition}
\begin{proof}
Let us denote by $x_{\CT,\pm}$ the two extremal points of $\CT$. The union of the two strong stable local leaves,  $W^{ss}_{\S}(x_{\CT,\pm})$ is referred to as  the stable boundary for $\S$. 

For small $\rho$, $W^{uu}_{\rho}(x)$ is well defined, due to Pesin theory. As $\mu$ is regular, and as $x$ is $\mu$-regular, a standard Borel-Cantelli argument shows that we may assume that  for any sufficiently big $t$, say $t> t_{0}$,$f_{-t}(W^{uu}_{\rho}(x))$ never intersects the stable boundary of $\S$. 

If this holds for some fixed $\rho$ and for any $t>t_{0}$, then, one can decrease and adjust $\rho'<\rho$ such that no $f_{-t}(W^{uu}_{\rho'}(x))$ intersects the stable boundary of $\S$ for $t\le t_{0}$. 

Doing like this, no $f_{-t}(W^{uu}_{\rho'}(x))$ intersects the stable boundary of $\S$ for any positive $t$. 

Moreover, $x$ has infinitely many negative return times by $F_{\S}$ because it is $\mu$-regular and $\Seps$ has positive $\mu$-measure. At any negative return time $-t$, $\cup_{t'\in[-t-\eps,-t+\eps]}f_{t'}W^{uu}_{\rho}(x)$ intersects $\S$ as a graph over $\CT$ (because it is transversal to $E^{ss}$) and does not intersect the stable boundary.
\end{proof}

\begin{remark}
\label{rem-Sucurve}
Note that Definition \ref{def-sucurveconti} also means/yields that $F_{\S}^{-n}$ of a $(\S,u)$-curve is also a $(\S,u)$-curve. 
$\blacksquare$\end{remark}

We denote by $\Lambda_{n,\Sigma}$ the set of all $(\Sigma,u)$-curves which are a graph over an interval of size  bigger than $1/n$ in $\CT$.

 \begin{proposition}
\label{prop-fakeunstablemani}
Assume that $\Lambda$ is a singular hyperbolic attractor of a three-dimensional vector field $X$. For any $S$-adapted cross-section $\Sigma$, 
\begin{enumerate}
\item $\bigcup_{n} \L_{n,\S}$ is dense in $\L\cap\S$,
\item $\bigcup_{n} \L_{n,\S}$ contains all the $F_{\S}$-periodic points,
\end{enumerate}
\end{proposition}
\begin{proof}
Note that $(2)\Rightarrow (1)$ because periodic points are dense (see Lemma \ref{lem-perioddense}). 

Let us prove point $(2)$. 
Let $x$ be a $\FS$-periodic point. Let $\mu$ be the $f_{1}$-invariant measure with support in the orbit of $x$, $\CO(x)$. It is an ergodic measure and it is non-singular because $\S$ does not contain singularities. Therefore it is a regular measure and thus ergodic. 

Because $x$ is periodic, the set $\{F^{n}_{\S}(x)\}$ has finite cardinality. The point $x$ is $\mu$-regular, $\mu$ is hyperbolic, hence one can construct a very small local unstable leaf $W^{uu}_{loc}(x)$ (possibly for $\kappa'_{0}\neq \ko$. One can adjust the size of this unstable local manifold such that for some small $\rho$, $\bigcup_{t\in[-\rho,\rho]}f_{t}(W^{uu}_{loc}(x))$ is a $cu$-surface included into $\Seps$.

The set of points of the form $F^{-n}_{\S}(x)$ is finite, and all of them are $\mu$-regular. Therefore they all admit some Pesin local unstable manifold, where the backward dynamics is contracting (due to sectional expansion in $E^{cu}$).  
Moreover, in our construction of $\S$, we assumed the extremal points of the basis are not in stable manifolds of periodic points. These two points yield that we can adjust the size of $W^{uu}_{loc}(x)$ and $\rho$ such that $\bigcup_{t\in[-\rho,\rho]}f_{t}(W^{uu}_{loc}(x))\cap \S$ is a graph  $W^{cu}_{\S}(x)$ over some small interval in $\CT$, 
for any $n$, $F^{-n}_{\S}(W^{uu}_{loc}(x))$ is well defined and is a graph over some interval of the basis. 

We have thus proved that  $W^{cu}_{\S}(x)$ is a $(\S,u)$-curve. 
\end{proof}
 
\begin{proposition}
\label{prop-descrilambda}
$\bigcup_{n} \L_{n,\S}$ is a collection of disjoint maximal $(\S, u)$-curves. 
\end{proposition}
\begin{proof}
For $x$ in $\bigcup_{n} \L_{n,\S}$, Zorn lemma shows that there exists a maximal $(\S, u)$-local unstable manifold which contains $x$. By definitions such a curve is a $(\S, u)$-curve. 
If $x$ belongs to two such maximal curves, then by definition of maximality none of the curve can be included into the other one and one must have at least one splitting (see Figure \ref{fig-split}): if each curve is the graph of a map , say, $G_{i}$, then one must find $y\in \CT$ such that $G_{1}(y)\neq G_{2}(y)$. 
\begin{figure}[htbp]
\begin{center}
\includegraphics[scale=0.5]{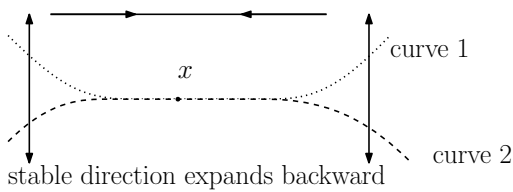}
\caption{Existence of splitting yielding contradiction}
\label{fig-split}
\end{center}
\end{figure}
If one consider the images by $F^{-1}_{\S}$, one these hand all these graphes are mapped into smaller and smaller graphes (due to Lemma \ref{lem-expandeuus}) and on the other hand, expansion in the stable direction yields that $F^{-n}_{\S}(y, G_{i}(y))$ must move away from each other. This is a contradiction.

 This shows that maximal $(\S,{u)}$-curves are disjoint.

\end{proof}

We finish this subsection with some technical lemma: 
\begin{lemma}
\label{lem-imagesucurve}
Let $W$ be a $(\S, u)$-curve. assume that $F_{\S}(W)$ is a piece of connected graph over some interval of the basis. Then, $F_{\S}(W)$ is also a $(\S, u)$-curve. 
\end{lemma}
\begin{proof}
Let $\CF^{cu}$ be a piece of $cu$-surface such that $W=\CF^{cu}\cap \S$. Any point of $\CF^{cu}$ can be written in a unique form as $f_{t}(y)$, with $y\in W$ and $t\in [-\eps,\eps]$ for some small $\eps$. In that case we set 
$$r(x):=r_{\S}(y).$$

Then, we consider $f_{r(.)}(\CF^{cu})$. The map $r_{\S}$ restricted to $W$ is continuous (because the image is a graph over an interval of the basis), thus $f_{r(.)}(\CF^{cu})$ is a $cu$-surface. By construction, its intersection with $\S$ is $\FS(W)$. 
\end{proof}


 \subsection{SACS with the GALEO property}
GALEO stands for Good Angle and Locally Eventually Onto. 
The LEO property is a crucial property to have some mixing and expansion  for the one-dimensional dynamics $g_{\CT}$. In addition to this, we also need to control the angle that any return of a small piece of $\CT$ makes with $\S$. This angle  with respect to $E^{ss}$ at return is bounded away from zero because the splitting is dominated and continuous. Nevertheless, there is no control inside $E^{cu}$. This is why we need to introduce this Good Angle property. We remind that $X^{\bot}$ stand for the orthogonal direction to $X$ in $E^{cu}$. 

\begin{definition}
\label{def-galeo}
The SACS $\S$ satisfies the GALEO property {\color{red} for $\kappa_1=\kappa_1(\Sigma)$} if for any open interval $I\subset \CT$, there exists $t$ such that $f_{t}(I)$ contains a curve, say $J$, satisfying 
\begin{enumerate}
\item $J$ as slope bounded by $\kappa_{1}$ as a curve from $X^{\bot}$ to $X$.
\item $J$ is included in some neighborhood  $\Seps$.
\item $\TCT\circ \iota (J)=\CT$. 
\end{enumerate}

\end{definition}

\begin{proposition}
\label{prop-leo}
The SACS $\S$ can be constructed such that it satisfies the GALEO property for  some $\kappa_1=\kappa_1(\Sigma)$
\end{proposition}
\begin{proof}
First, we consider some $\mu_{SRB}$-generic point $x_{0}$. We assume it has a ``long'' unstable local leaf and that it is an accumulation point for points having long unstable leaves with the same angle (up to some very small variations) with respect to the flow direction $X$. We denote by $\CA$ this set of points. 
Note that if $x$ is such a point then $f_{t}(x)$ is also such a point. 
All the points we may consider here are $\mu_{SRB}$-generic, thus we may assume that Lebesgue almost all the points in their local unstable leaves also are  $\mu_{SRB}$-generic. As this property is invariant along orbits, 
this yields a collection, say $\CB$, of $E^{cu}$ surfaces, foliated by local unstable leaves and flow trajectories, for $t\in[-\eps,\eps]$ and for which Lebesgue almost every point is $\mus$-generic. 

Then, very close to $x_{0}$ we pick a periodic point, say $Q$. It is regular with respect to the regular measure supported by this periodic orbits, thus $Q$ admits a piece of local unstable leaf. We only consider a very small piece of this local unstable leaf, and construct the SACS with this basis as it is done in the proof of Prop. \ref{prop-goodcrosssec}. Note that $\CT$ has to be small compared to the ``long'' unstable leaves of $\mu_{SRB}$-generic points we have just considered above. 

\medskip
We now show that this SACS has the GALEO property. For that we consider a small interval $I\subset \CT$. Note that Lebesgue almost every point in $I$ belongs to the stable leaf of $x\in\CA$. Consequently, this means that $I$ can be considered as the projection on $\CT$ of the identification  (the map $\iota$) of a piece of some  true local unstable leaf $W$ 
inside some $E^{cu}$-surface in $\CB$. 

\begin{figure}[htbp]
\begin{center}
\includegraphics[scale=0.5]{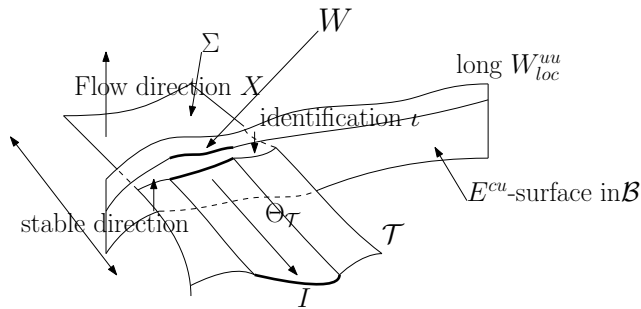}
\caption{Construction of a SACS with the GALEO property}
\label{Fig-galeo}
\end{center}
\end{figure}

Lebesgue almost every point in $W$ belongs to $\CA$ and will return infinitely many often in $\CA$.This yields that for infinitely many $t$'s, $f_{t}(I)$ is a curve  in $\Seps$ with fixed angle with respect to the flow direction $X$. Any upper bound for the angle defines $\kappa_{1}$.

The sectional expansion shows that the sizes of these images must increase (exponentially with respect to the return times). This shows that the images of $I$ by $F_{\S}$ will eventually cross $W^{ss}_{\S}(Q)$. At that moment we have small interval $I'=:f_{t}(J)$ for some small interval $J\subset I$ crossing $W^{ss}_{\S}(Q)$. Then, the $\l$-lemma shows that the forward images $f_{t+nt_{0}}(J)$ (if $t_{0}$ is the period for $Q$) must accumulate on $W^{uu}_{Q}$, thus for some iterate, $g_{\CT}^{n}(J)=\CT$.  
This prove the LEO  part for the GALEO property. 

\medskip
To prove the  GA part, we only have to consider returns of $W$ in $\CA$.  

 \end{proof}

\begin{corollary}
\label{coro-leo}
Any point $x\in \CT$ has a dense set of preimages for $g_{\CT}$. 
\end{corollary}

As $\CT$ is compact, we also have:
\begin{corollary}
\label{cor-periodense1dim}
Periodic points are dense for $g_{\CT}$.
\end{corollary}

\section{Construction of a Mille-Feuilles}\label{sec-millefeuilles}
 \subsection{First step. A rectangle with Markov property}

We consider a SACS $\S$  with the GALEO property. We look at the first return map $F_{\S}$ into $\S$. We re-employ notations from above,  $F_{\S}(x)=f_{r_{\S}(x)}(x)$. 

The constant $\kappa_{1}$ involved in the GALEO property is thus fixed. 

\subsubsection{Vertical band and eligible curves of type 1}

Lemma \ref{lem-perioddense} immediately yields:

\begin{definition}
\label{def-strip}
A strip in $\S$ is a set $\Theta_{\CT}^{-1}[p_{-},p_{+}]$ in $\S$ where $p_{\pm}$ are two consecutive (for the natural order relation)  points of some $g_{\CT}$ periodic orbit. We call interior of the strip the set $\Theta_{\CT}^{-1}]p_{-},p_{+}[$. 
\end{definition}

In the rest of the paper, one shall talk about strips defined by periodic points $p_{\pm}$ or equivalently of the strip over $[p_{-},p_{+}]$.

\begin{definition}
\label{def-eligible curve type 1}
 Let $\CB$ be the strip over $[p_{-},p_{+}]$ in $\S$.  An eligible curve of type 1 in the strip is 
 the intersection of the strip with a maximal $(\S, u)$-curve crossing over the strip $\CB$. 
\end{definition}

Let $\CF^{u}$ be an eligible curve of type 1 in the strip $\CB$ over $[p_{-},p_{+}]$.
By definition of maximal $(\S,u)$-curve and by definition of the strip, we emphasize that for every $n$, $F_{\S}^{-n}(\CF^{u})$ is a piece of $(\S,u)$-curve.

As $p_{\pm}$ are two consecutive points of a periodic orbit, this yields that $F_{\S}^{-n}(\CF^{u})$ is either totally outside the strip $\CB$ or totally contained in it. Otherwise, there would be some image of $p_{+}$ or $p_{-}$ (by $g_{\CT}$) between these two points.

Furthermore, for any point $x$ in $\CF^{u}$, one can  define the sequence of backward return times, that is 
$$F_{\S}^{-n}(x):=f_{-r^{-n}(x)}(x).$$

\subsubsection{Eligible curve of type 2}
We emphasize (and remind) that  a $(\S,u)$-curve cannot necessarily be written under the form $\iota(W^{uu}_{loc})$. 
This justify the next definition.

\begin{definition}
\label{def-eligible curve type 2}
Let $\kappa_{2}>\kappa_{1}$ be fixed. 
 Let $\CB$ be the strip over $[p_{-},p_{+}]$ in $\S$.  An eligible curve of type 2 in the strip is an eligible curve of type 1, say $\CF^{u}$ such that 
 \begin{enumerate}
\item there exists a local $\ko$-unstable manifold $W^{uu}_{loc}(x)$ contained in a small neighborhood $\Seps$ such that $\iota(W^{uu}_{loc}(x))=\CF^{u}$,
\item As a graph from $\S$ to $<X>$, $W^{uu}_{loc}(x)$ has a slope bounded by $\kappa_{2}$. 
\end{enumerate}
\end{definition}

We emphasize that if the basis of the SACS is  obtained from a periodic point (as in the construction of Prop. \ref{prop-leo})
then, the segment $[p_{-},p_{+}]$ (of the basis) is an eligible curve of type 2 because it is a true piece of strong unstable leaf and it is inside $\S$ thus a slope equal to 0.

\begin{definition}
\label{def-millefeuilles}
 Let $\CB$ be the strip over $[p_{-},p_{+}]$ in $\S$.  
 The pre-mille-feuilles associated to $\CB$ is the union $\CP\CM$ of eligible curves of type 1.

We set $W^{ss}_{loc}(x):=W^{ss}_{\S}(x)\cap\CP\CM$ and $W^{uu}_{\CP\CM}(x):=\CF^{uu}(x)\cap\CP\CM$, where $\CF^{uu}(x)$ is the eligible curve of type 1 which contains $x$. 

We call interior $\grosinte{\CP\CM}$ of the pre-mille-feuilles $\CP\CM$ the set $\CP\CM\cap\Theta_{\CT}^{-1}]p_{-},p_{+}[$.

\end{definition}

\begin{remark}
\label{rem-increask2}
We emphasize that the bigger $\kappa_{2}$ is the fatter the pre-mille-feuilles is. The same holds if $\kappa_{0}$ increases. 
$\blacksquare$\end{remark}


\begin{lemma}
\label{lem-millecomp}
Any pre-mille-feuilles is compact.
\end{lemma}
\begin{proof}
By construction, any eligible curve is closed thus compact. 
Let us now consider a family of eligible curves  accumulating themselves on some curve $\CF^{u}$. 
All the eligible curves are graph of Lipschitz maps over $\I$ with bounded slope. This holds because they are all restrictions to $\S$ of a $cu$-surface and $E^{cu}$ is transversal to $E^{ss}$. Therefore, the limit curve  $\CF^{u}$ is also a graph of a Lipschitz map over $\I$. By continuity of the splitting $E^{cu}$, $\CF^{u}$ is a $(\S,u)$-curve. It is thus an eligible curve of type 1. 
\end{proof}

A pre-mille-feuilles has a canonical product structure: if $x$ and $y$ are in $\CP\CM$ then,  the eligible curve of type 1 containing $x$  is transversal to $W^{ss}_{loc}(y)$ and they intersect themselves at a unique point, say $z$. We set 
\begin{equation}
\label{eq-crochet}
z:=[x,y].
\end{equation}

\begin{lemma}
\label{lem-markovreturn}
A pre-mille-feuilles has Markov return map: 
If $x\in\grosinte{\CP\CM}$ and $F_{\S}^{k}(x)\in\grosinte{\CP\CM}$, then 
\begin{enumerate}
\item $F_{\S}^{k}(W^{ss}_{\CP\CM}(x))\subset W^{ss}_{\CP\CM}(F_{\S}^{k}(x))$.
\item $F_{\S}^{-k}(W^{uu}_{\CP\CM}(F_{\S}^{k}(x)))\subset W^{uu}_{\CP\CM}(x)$. 
\end{enumerate}
\end{lemma}
\begin{proof}
$\bullet$ Let $k$ be a return time for $x$. 
Then, $W^{uu}_{\CP\CM}(\FS^{k}(x))$ crosses over the band $\CB$ and there are two intersection points with the boundary, say $y_{\pm}$. One can consider the preimages of these two points (because they belong to a $(\S,u)$-curve) and we set 
$$z_{\pm}:=\FS^{-k}(y_{\pm}).$$
We consider the band $\CB':=\TCT^{-1}[\TCT(z_{-}),\TCT(z_{+})]$. 
By construction this is a band strictly inside $\CB$ and then all the eligible curves in $\CP\CM$ cross over $\CB'$.

$\bullet$ We can consider the images by the flow of all these eligible curves restricted to $\CB'$. By construction, their images by $\FS^{k}$ are curves which do cross over the band $\CB$. By Lemma \ref{lem-imagesucurve}, they are $(\S,u)$-curves.

In other words, these images curves are pieces of maximal $(\S,u)$-curves and do cross-over $\CB$. They are eligible curves of type 1.

%
This concludes the proof of the lemma. 
\end{proof}

\subsection{The mille-feuilles}
One of the problem of the pre-mille-feuilles is that first returns are not necessarily with good dynamical properties. In particular we will need to control the distorsion on return times: if $x$ and $y$ belong to $[p_{-},p_{+}]$, if $x'$ and $y'$ are respective preimages on the same inverse branch, then we shall need to control $r_{\S}(x)-r_{\S}(y)$. This is possible if we can control le slope of the image of the basis associated to the inverse branch (see below). This is why we need to remove horizontal strips in the pre-mille-feuilles.

\subsubsection{Step one. Construction of a mille-feuilles from a Pre-mille-feuilles}
Let us consider a pre-mille-feuille $\CP\CM$ with transversal $\CT$. 

The return map defines vertical bands and horizontal strips. Each vertical band, say $\CB'$, is mapped onto an horizontal strip $\CH$ which crosses over $\CB$ by the return map $\FS^{k}$. 

Hypothetically, we have countably many of them. We denote them by $(b_{i})$ and $(h_{i})$. 

\begin{lemma}
\label{lem-horoemptyinter}
Two different horizontal bands are disjoint. 
\end{lemma}
\begin{proof}
See picture \ref{pic-band-cantor}. 

Let $x$ be in $\grosinte{\CP\CM}$ and $k$ be such that $F_{\S}^{k}(x)\in\grosinte{\CP\CM}$. Let $\CB'$ be the vertical band as above. By construction $\FS^{k}(\CB')$ is an horizontal strip $\CH$ which crosses over $\CB$. More precisely, as $\S$ is a $s$-adapted, there exists two horizontal bands with empty intersection with $\L$ which separate $\CH\cap\L$ from the rest of $\L$ . 

\begin{figure}[htbp]
\begin{center}
\includegraphics[scale=0.5]{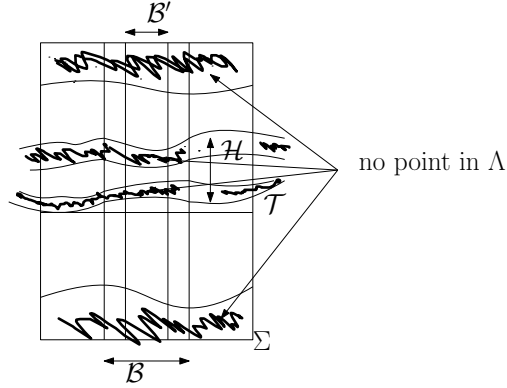}
\caption{Strip with empty intersection with $\L$}
\label{pic-band-cantor}
\end{center}
\end{figure}
\end{proof}

Roughly speaking, in the vertical direction, $\CP\CM$ has a totally disconnected structure and there are two ``empty'' strips which separate $\CH$ from the rest of $\CP\CM$. 

If $b_{i}$ is a vertical band with return $k$, by construction the image of the transversal $\FS^{k}(\CT\cap b_{i})$ is an eligible curve of type 1 in $\CP\CM$ (see Lemma \ref{lem-imagesucurve}). 

\begin{definition}
\label{def-codealphabeta}
Let $\kappa_{2}>\kappa_{1}$ be fixed. 
We say that $b_{i}$ and $h_{i}$ are of type $\al$ if $\FS^{k}(\CT\cap b_{i})$ is an eligible curve of type 2. 
Otherwise we say it is of type $\be$. 

The mille-feuilles associated to $\CP\CM$ and $\kappa_{2}$ is the collection of eligible curves of type 1 included in the horizontal strips of type $\al$ (see Fig. \ref{Fig-mille-feuilleconstruct}). 
\end{definition}

\begin{figure}[htbp]
\begin{center}
\includegraphics[scale=0.5]{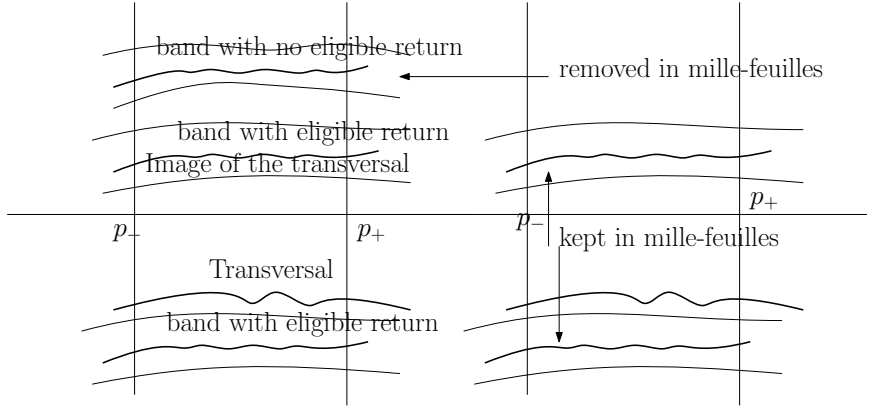}
\caption{Remove of keep bands with respect to eligible returns}
\label{Fig-mille-feuilleconstruct}
\end{center}
\end{figure}

By construction, $\CM$ can be written as the image $\Phi_{\CM}([-1,1]\times\Gamma)$, where $[-1,1]$ is identified to the horizontal $[p_{-},p_{+}]$ and $\Gamma$ is the Cantor set giving the eligible curves in $\CM$. Each $\CM_w:=\Phi_{\CM}([-1,1],w)$ is one eligible curve of $\CM$. They will be referred to as horizontals. Verticals are pieces of strong stable leaves. 

In other words, the mille-feuilles $\CM$ has a rectangle structure. Because entire blocks of returns (the bands $b_{i}$) have been removed it is still compact. 

\subsubsection{Step 2. First return $F_{\CM}$ is  Markovian}
The mille-feuilles has been defined from the pre-mille-feuilles. Even if the horizontal strips of type $\be$ have been removed, the associated vertical bands are still well defined in $\CM$.

We denote by $F_{\CM}$ the first return map in $\CM$. 
\begin{lemma}
\label{lem-fcm}
The first return map $F_{\CM}$ corresponds to the first return into a horizontal strip of type $\al$ by iterations of $\FS$. 
\end{lemma}
\begin{proof}
Because $\CM\subset\CP\CM$ any return in $\CM$ is a return in $\CP\CM$. As $\CM$ is the restriction of $\CP\CM$  of horizontal strips of type $\al$, the lemma holds. 
\end{proof}
This means that if $x$ eventually has a return  into a band of type $\al$, and if $x$ belongs to $b_{i}$ of type $\be$ with return $k_{1}$ and is such that $\FS^{k_{1}}(x)$ belongs to a vertical band  $b_{j}$ of type $\be$ and so on up to the $(n+2)^{th}$ return: 
$$x\in b_{i_{0}}\text{ type }\be{\longrightarrow}_{\FS^{k_{1}}}b_{i_{1}}\text{ type }\be{\longrightarrow}_{\FS^{k_{2}}}\ldots \longrightarrow b_{i_{n}}\text{ type }\be{\longrightarrow}_{\FS^{k_{n+1}}}h_{i_{n}}\cap b_{i_{n+1}}\text{ type }\al \text{ return }k_{n+2},$$
then,  $F_{\CM}(x):=\FS^{k_{1}+\ldots +k_{n+2}}(x)$. It is well-defined on 
$$b_{i_{0}}\cap \FS^{-k_{1}}(((b_{i_{1}}\cap \FS^{-k_{2}}((b_{i_{2}})\ldots \FS^{-k_{n+1}}(b_{i_{n+1}}))).$$

Doing like this, we define a countable collection of bands that are mapped on strips by $F_{\CM}$. All these bands define a countable collection of intervals $(I_{n})$ on $[p_{-},p_{+}]$. 
Moreover, intervals $I_{n}$ have disjoint interiors.

Again, for $x\in \CM$, we set $W^{uu}_{\CM}(x)=W^{uu}_{\CP\CM}(x)$ and $W^{ss}_{\CM}(x)=W^{ss}_{loc}(x)\cap\CM$. 
\begin{proposition}
\label{prop-Markov return}
The map $F_{\CM}$ is Markov: 
\begin{enumerate}
\item $F_{\CM}(W^{ss}_{\CM}(x))\subset W^{ss}_{\CM}(F_{\CM}(x))$.
\item $F_{\CM}^{-1}(W^{uu}_{\CM}(F_{\CM}(x)))\subset W^{uu}_{\CM}(x)$. 
\end{enumerate}
\end{proposition}
\begin{proof}
Property $(2)$ just follows from Lemma \ref{lem-markovreturn}. Property $(1)$ is less obvious because points have been removed from $\CP\CM$ thus also from the images. 

The set $W^{ss}_{\CM}(x)$ is obtained from $W^{ss}_{\CP\CM}(x)$ by removing intervals corresponding to horizontal strips of type $\be$. This means that there are less points in $W^{ss}_{\CM}(x)$ than in $W^{ss}_{\CP\CM}(x)$. Now, $F_{\CM}(W^{ss}_{\CM}(x))$ is inside an horizontal strip of type $\al$ (in $\CP\CM$ or in $\CM$), where we have not removed points as we built $\CM$ from $\CP\CM$. This shows that $(1)$ holds because it holds for $W^{ss}_{\CP\CM}(x)$. 
\end{proof}

This yields that  $F_{\CM}$ is a skew product. A point of $\CM$ can be represented as $z=\Phi_{\CM}(x,y)$ with $(x,y)\in [-1,1]\times\Gamma$. Then, $F_{\CM}(z)=:(g_{\CM}(\TCT(x)), G_{\CM}(z))$. 
Inverse branches  for $g_{\CM}$ are well-defined.


\subsubsection{Step 3. Return times are Dynamically H\"older}

\begin{definition}
\label{def-dawei-dynahold}
A function $\varphi:\CM\to{\mathbb R}$ is said to be dynamically H\"older if
it is constant along fibers $\TCT^{-1}(\{x\})$ and there exist $\kappa$ and $\ga$ such that for every $x$ and $y$ in the same $n$-cylinder
$$\left|\sum_{k=0}^{n-1}\varphi\circ g_{\CM}^{k}(x)-\varphi\circ \gm^{k}(y)\right|\le \kappa.|g_{\CM}^{n}(x)-\gm^{n}(y)|^{\ga}.$$
\end{definition}

\begin{proposition}
\label{prop-returndynlip}
If we set $F_{\CM}(x)=f_{r_{\CM}(x)}(x)$, then the return-time map $r_{\CM}$ is Dynamically H\"olders. 
\end{proposition}

The proof is postpone for later. We first need to introduce some more vocabulary.

Remind that there is an ``horizontal'' reference segment $[p_{-},p_{+}]$. $F_{\CM}$ is the first-return map in $\CM$ and we have 
$$F_{\CM}(x):=\FS^{\tau_{\CM}(x)}(x).$$
Furthermore, we set $F_{\CM}(x)=f_{r_{\CM}(x)}(x)$. By construction, $g_{\CM}:=\Theta_{\CT}\circ \FM=\TCT\circ F_{\CM}$. 

The map $g_{\CM}$ may be not well defined everywhere on $[p_{-},p_{+}]$ and it can also be multi-valued on some points because we focused on point returning infinitely many times into the pre-mille-feuilles and eventually returning in horizontal strips of type $\al$. 

\begin{definition}
\label{def-cylindersG1}
A generation 1 cylinder $C_{1}$ in $\I$ is an interval $[a,b]$ (in $\CF$) such that for some point $x\in]a,b[$, $F_{\CM}(x)$ belongs to $\inte\CM$. If $F_{\CM}(x)=F^{\tom(x)}_{\S}(x)$, then $\tom(x)$ is called the return time for the cylinder. 
\end{definition}
\begin{remark}
\label{rem-returncylinder}
Note that  for any $y\in]a,b[$, $\tom(y)=\tom(x)$. Moreover,   $g_{\CM}(C_{1})=\I$.
$\blacksquare$\end{remark}

Each 1-cylinder defines an inverse branch for $\gm$. This allows to define higher generation cylinders:

\begin{definition}
\label{def-cylinderGn}
We define by induction cylinders of generation $n+1$, as the preimages by an inverse branch of $\gm$  of cylinders of generation $n$. 
A cylinder of generation $n$ will also be called $n$-cylinder. 

For $x$ satisfying $g_{\CM}^{n}(x)\in\inte\CM$, the $n$-cylinder which contains $x$ will be denoted by $C_{n}(x)$. 
\end{definition}

Associated to the cylinders of generation $n$, there is a $n^{th}$-return time $\tom^{n}(x)$ (for $x$ in its interior). 
Note the cocycle relation 
\begin{equation}
\label{equ-cocuclereturn}
\tom^{n+1}(x)=\tom(x)+\tom^{n}(g_{\CM}(x)). 
\end{equation}

\bigskip
\paragraph{We can now prove Proposition \ref{prop-returndynlip}}

First we do two simple observations. 

\begin{obs}\label{obs-deltadistoelig2}
Let $\CF^{u}$ be an eligible curve of type 2. Set $\CF^{u}=\iota(\mathcal W^{u})$ where $\mathcal W^{u}$ is real local strong unstable leaf. Let $A$ and $B$ be in $\mathcal W^{u}$, $A'=\iota (A)$, $B'=\iota (B)$, $A''=\TCT(A')$ and $B''=\TCT(B')$ (see Fig. \ref{fig-controldistotype2}).

Then there exists constant $\kappa=\kappa(\kappa_{1},\kappa_{2})>0$ and $\ga>0$ such that 
$$d_{\CW^{u}}(A,B)\le \kappa d_{\CT}^{\ga}(A'',B'')\text{ and }d_{\CT}(A'',B'')\le \kappa d_{\CW^{u}}^{\ga}(A,B).$$
\end{obs}
This holds because $\CW^{u}$ has bounded slope with respect to $\S$ and $X$, the eligible curve has bounded slope with respect to $\CT$ and $E^{ss}$ and the stable holonomies are H\"older continuous. 

The second observation is a simple consequence of Lemma \ref{lem-expandeuus}: 
\begin{obs}\label{obs-expan}
If $\CW^{u}$ is the image by $F_{\S}^{k}$ of some connect $(\S,u)$-curve, and if we set $A=\FS(a)$ and $B=\FS(b)$, then 
$$d_{\CT}(a,b)\le \kappa e^{-kR}d_{\CW^{u}}(A,B),$$
where $R$ is the minimal return time in $\S$. 
\end{obs}

\begin{figure}[htbp]
\begin{center}
\includegraphics[scale=0.6]{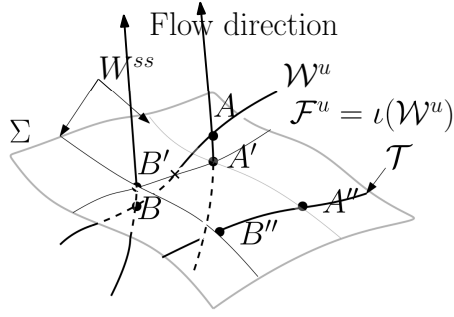}
\caption{Control of distorsion for eligible curve of type 2}
\label{fig-controldistotype2}
\end{center}
\end{figure}

We remind that $\iota$ is Lipschitz continuous  because the vector field $X$ is $\CC^{2}$ and $\TCT$ is H\"older continuous (see \cite{AraujoMelbourne19}). 


The proof is done by induction. 
We consider $x$ and $y$ in the same $1$-cylinder  $C_{1}$ in the transversal $[p_{-},p_{+}]$ (see Fig. \ref{fig-deltareturn}). By definition of $\CM$, $\FM(C_{1})$ is an eligible curve of type 2. Then, Observation \ref{obs-deltadistoelig2} yields

\begin{figure}[htbp]
\begin{center}
\includegraphics[scale=0.5]{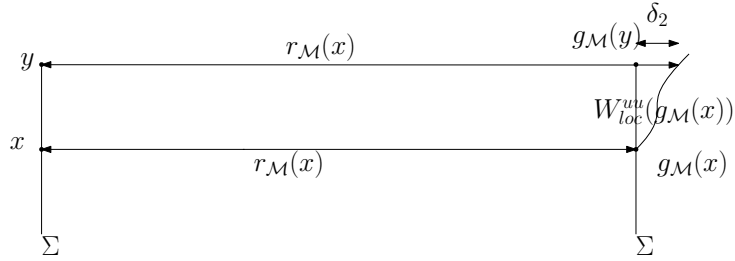}
\caption{Difference of return times}
\label{fig-deltareturn}
\end{center}
\end{figure}

\begin{eqnarray*}
|r_\CM(x)-r_\CM(y)|&=& |\delta_{2}|\\
&\le &\kappa.d^{\gamma}_{\CT}(\gm(x),\gm(y)).
\end{eqnarray*}

Moreover, Observation \ref{obs-expan} yields 
\begin{equation}
\label{eq-contraT}
d_{\CT}(x,y)\le e^{-R}d^{\ga}_{\CT}(\gm(x),\gm(y)).
\end{equation}

Assume that 
$$|\rm^{n}(x)-\rm^{n}(y)|\le C d_{\CT}^{\ga}(\gm^{n}(x)-\gm^{n}(y))$$
holds whenever $x$ and $y$ are in the same $n$-cylinder and let us prove the same property for $(n+1)$-cylinder.
We pick $x$ and $y$ in the same $(n+1)$-cylinder. Then they are in the same $n$-cylinder and furthermore,  $\gm^{n}(x)$ and $\gm^{n}(y)$ are in the same 1-cylinder. Moreover 
$$\rm^{n+1}(x)=\rm^{n}(x)+\rm(\gm^{n}(x))$$ and the same holds for $y$. 

Then we get 
\begin{eqnarray*}
|\rm^{n+1}(x)-\rm^{n+1}(y)|&\le& |\rm^{n}(x)-\rm^{n}(y)|+|\rm(\gm^{n}(x))-\rm(\gm^{n}(y))|\\
&\le & C. d_{\CT}^{\ga}(\gm^{n}(x)-\gm^{n}(y))+ \kappa. d^{\ga}_{\CT}(\gm ^{n+1}(x),\gm^{n+1}(y))\\
&\le & C.e^{-\ga.R}d^{\ga^{2}}_{\CT}(\gm^{n+1}(x),\gm^{n+1}(y))+\kappa. d^{\ga}_{\CT}(\gm ^{n+1}(x),\gm^{n+1}(y)). 
\end{eqnarray*}
If $C$ satisfies $Ce^{-\ga.R}L^{\ga}+\kappa\le C$, with $L$ equal to the length of $\I$, then the same property holds at stage $n+1$ and the proposition is proved.

\subsubsection{Step 4. Cylinders are dense in $\I$}

 \begin{proposition}
\label{prop-densereturnlillefeuille}Let $\kappa_{2}$ be fixed. Let $\CM$ be a mille-feuilles.
  The set of points with infinitely many returns has dense projection in $\I$.

\end{proposition}
%
\begin{proof}
$\bullet$ In the first step, we prove denseness of points with at least one return time.
This is direct consequence of the GALEO property. 
If $I$ is a small interval in the transversal $\I$, then, for some return $\FS^{k}(I)$ is an eligible curve of  type 2 because $\kappa_{1}<\kappa_{2}$. Moreover it is a long curve whose projection by $\TCT$ overlaps $\CT$, and then overlaps $\I$. In other word, this return is return for $\FM$, and then $I$ contains points with at least one return in $\CM$. 

Furthermore, the Markov property (see Lemma \ref{lem-markovreturn}) yields that 
the set of points with at least one return into $\inte\CM$ has an open  projection in $\I$.

$\bullet$ We finish the proof of the Proposition. 
Points in the transversal with at least one return form an open and dense set in $\CI$. 
To employ vocabulary of the one-dimensional dynamics associated to $\CM$, we have just proved here that the union of the interiors of the 1-cylinders in $\I$ is open and dense in $\I$. 

It is thus immediate that the union of the interiors of 2-cylinders is open and dense in each 1-cylinder, because the return  map a 1-cylinder onto $\I$. Consequently, and by induction, for every $n$, the union of the interiors of the $n$-cylinders is open and dense in $\I$. By Baire's Theorem, its intersection is dense. This finishes the proof of the proposition. 
\end{proof}

%

\subsection{Generalized Mille-feuilles and invariant measures}

We finish this section with two important points to prove uniqueness of the equilibrium state among regular measures. 
\begin{proposition}
\label{prop-millefeuillesmeasure}
For any regular measure $\mu$, for any $\mu$-regular point $x$, there exists a mille-feuilles $\CM$ containing $x$ and with positive* $\mu$-measure. 

Furthermore, it can be constructed such that it also have positive* $\mu_{SRB}$-measure. 
\end{proposition}
\begin{proof}
We have already seen in Proposition \ref{prop2-goodcrosssec} that we can construct a SACS with this property. 
The main point is to check that we can also get the GALEO property. This holds because to get the GALEO property (see Prop. \ref{prop-leo}) we used $\mu_{SRB}$-regular points but we could actually use $\mu$ (which is Hyperbolic ans thus has \AE Pesin local unstable leaves). 

Then, denseness of $g_{\CT}$-periodic point allows to choose $\I$ such that $x$ belongs to the pre-mille-feuilles, and if $\kappa_{2}$ increases, $x$ will belong to the mille-feuilles. If $x$ is a density point for $\mu$ with all the (finitely many) properties involved  above, then the mille-feuilles has positive* $\mu$-measure.

Now, $\mu_{SRB}$ has dense support and any $\mu_{SRB}$ regular point admits a local unstable manifold (due to Pesin theory). Therefore, we may increase $\kappa_{2}$ such that every $\mu_{SRB}$-regular point sufficiently close to $x$ has a so long unstable local manifold that it is eligible of type 2. 
\end{proof}

Now, we introduce the concept of generalized mille-feuilles: 

\begin{definition}
\label{def-generalizedmilfeui}
Let $\CM$ be a mille-feuilles with basis $\I$. Let $C_{n}$ be a $n$-cylinder $\CB_{n}:=\TCT^{-1}(C_{n})\cap \CM$. Then $\CB_{n}$ is called a generalized mille-feuilles. 
\end{definition}
A generalized mille-feuilles is not properly speaking a mille-feuilles because the extremal points in the basis are not periodic but pre-periodic (for the return global map $g_{\CT}$). Nevertheless, the crucial dynamical properties are the same: 

\begin{enumerate}
\item It has a rectangle structure, and can be seen as $[0,1]\times \Gamma$ where $\Gamma$ is a Cantor set in $[0,1]$. Verticals $\{x\}\times \Gamma$ are Cantor sets into $W^{ss}_{loc}$ and horizontals ``$[0,1]\times \{y\}$'' are the restriction of $W^{uu}_{\CM}(y)$ to the vertical band $\TCT^{-1}(C_{n})$. 
 \item It is compact because $\CM$ and $C_{n}$ are compact. 
\item The first return is Markov: image of verticals are strictly inside verticales and images of horizontals overlap horizontals. In other words, Prop. \ref{prop-Markov return} holds.
\item Return times are dynamically H\"older, that is Prop. \ref{prop-returndynlip} holds because returns in the generalized mille-feuilles are returns in the mille-feuilles. 
\end{enumerate}

\section{Inducing scheme over a mille-feuilles}\label{sec-inducingschemealanono}

\subsection{Induced potential}
Let $\CM$ be a mille-feuille. We recall equalities:
$$\gm^{n}(x):=\TCT\circ\FS^{\tom^{n}(x)}(x)=\TCT\circ f_{r_\CM^{n}(x)}(x).$$
In the following, $\tom$ will be referred to as the return time for $\FS$ and $r_{\CM}$ to as the \emph{roof function}. 

\begin{nota}
To lighten notations one shall write $\disp |x-x'|$ instead of $d_{\CT}(\gm(x),\gm(x'))$.
\end{nota}

\begin{definition}
Assume that $\CM$ is a mille-feuilles of a singular hyperbolic attractor $\Lambda$ of $X$ with  roof function $r_\CM$. 
For any potential $V:\Lambda\to \R$, the function $V_{\CM,Z}$ defined on $\CM$ by 
$$V_{\CM,Z}(x):=\int_{0}^{r_\CM(x)}V\left(f_{t}(x)\right)\,dt-Zr_\CM(x)$$
 is said to be the \emph{induced potential of $V$ associated to parameter $Z$}. 
\end{definition}

\begin{lemma}
\label{lem-coboundaryA}
For $y$ be in $\CM$ set $B(y):=\disp \int_{0}^{+\8}V\left(f_{t}(y)\right)-V\left(f_{t}(\TCT(y))\right)\,dt$ and $W(y):=\vmo(y)-B(f_{\rm(y)}(y))$. 
Then, 
\begin{equation}
\label{equ-coboundaryAW}
\vmo(y)=W(\TCT(y))+B(y)-B\circ \FM(y).
\end{equation}
\end{lemma}
\begin{proof}
This is a standard computation. Set $\TCT(y)=x$. 
\begin{eqnarray*}
\vmo(y)&=& \int_{0}^{\rm(y)}V\circ f_{t}(y)dt\\
&=& \int_{0}^{\rm(y)}V\circ f_{t}(x)dt+\int_{0}^{\rm(y)}V\circ f_{t}(y)dt-\int_{0}^{\rm(y)}V\circ f_{t}(x)dt\\
&=& \int_{0}^{\rm(x)}V\circ f_{t}(x)dt+\int_{0}^{+\8}V\circ f_{t}(y)-V\circ f_{t}(x)dt-\int_{\rm(y)}^{+\8}V\circ f_{t}(y)-V\circ f_{t}(x)dt\\
&=& \vmo(x)+B(y)-\int_{0}^{+\8}V\circ f_{t}(f_{\rm(y)}(y))-V\circ f_{t}(\TCT(f_{\rm(y)}(y)))dt \\
&&\hskip 4cm   - \int_{0}^{+\8}V\circ f_{t}(f_{\rm(x)}(x))-V\circ f_{t}(\TCT(f_{\rm(x)}(x)))dt\\
&=& \vmo(x)+B(y)-B\circ\FM(y)-B\circ\FM(x)\\
&=& W(x)+B(y)-B\circ\FM(y).
\end{eqnarray*}
\end{proof}

\begin{remark}
\label{rem-bbounded}
We emphasize that uniform contraction in the strong stable direction yields that $B$ is well-defined and uniformly bounded. 
$\blacksquare$\end{remark}


%

One of the main point in our proof is to find good Banach spaces one which the transfer operator will act. For that purpose, we give here a key proposition:
\begin{proposition}
\label{prop-regularity}
Assume $V$ is $\al$-H\"older. Then there exists $\gamma>0$ and $C>0$ such that
if $x$ and $x'$ are in the same 1-cylinder, then 
$$|W(x)-W(x')|\le C |\gm(x)-\gm(x')|^{\ga}$$
holds, where $\ga$ comes from Prop. \ref{prop-returndynlip}. 
\end{proposition}

\begin{proof}

%
%
%
%

$\bullet$
We want to bound $\disp |W(x)-W(x')|$ with  respect to $\disp |\gm(x)-\gm(x'|$. Note that  Prop. \ref{prop-returndynlip} yields
$$|\rm(x)-\rm(x')|\le \kappa|\gm(x)-\gm(x'|^{\ga}.$$
Note that we can always decrease $\ga$, the same kind of inequality will still hold. 
On the other hand we have 
\begin{eqnarray*}
\vmo(x)-\vmo(x')&=& \int_{0}^{\rm(x)}V(f_{t}(x))dt-\int_{0}^{\rm(x')}V(f_{t}(x'))dt\\
&=& \int_{\rm(x')}^{\rm(x)}V(f_{t}(x))dt+\int_{0}^{\rm(x')}V(f_{t}(x))-V(f_{t}(x'))dt. 
\end{eqnarray*}
As $V$ is bounded,  the first summand in the last equality is upper bounded by some quantity of the form  $\kappa'|\gm(x)-\gm(x'|^{\ga}$. 

Then, we recall that $\CT$ is a true piece of unstable leaf. Therefore, by definition of a true local unstable leaf (see Def. \ref{def-lambdacunstableman}) $\al$-H\"older regularity for $V$ shows that
the second summand in the last equality is upper bound by some quantity of the form  
$$\kappa''d^{\al}(f_{\rm(x')}(x),f_{\rm(x')}(x')).$$  
Then, we use Observations \ref{obs-deltadistoelig2} and \ref{obs-expan} to get 
$$d(f_{\rm(x')}(x),f_{\rm(x')}(x'))\le \kappa'''|\gm(x)-\gm(x')|^{\ga}.$$
Finally, 
$$|\vmo(x)-\vmo(x')|\le \wh\kappa |\gm(x)-\gm(x'|^{\ga}$$
holds for some $\wh\kappa$ independent of $x$ and $x'$ and if we adjust the magnitude of $\ga$.

$\bullet$
Let us now give a  bound for $|B(\FM(x))-B(\FM(x')|$. For simplicity we set $r=\rm(x)\le \rm(x')=:r'$ and $\DR:=r'-r$. 
$T$ and $\eps$ are parameter. On figure \ref{Fig-Bhold} $y$, $y'$, $z$ and $z'$ stand for $\FM(x)$, $\FM(x')$, $\gm(x)$ and $\gm(x')$. 
Note that by construction of $\S$, $y$ and $z$ on one hand, $y'$ and $z'$ on the other hand are in the same strong stable leaf (for the flow). Furthermore $z$ and $z'$ are in the same strong unstable leaf (the transversal $\CT$) but $y$ and $y'$ do not necessarily lie in the same unstable leaf. 

\begin{figure}[htbp]
\begin{center}
\includegraphics[scale=0.5]{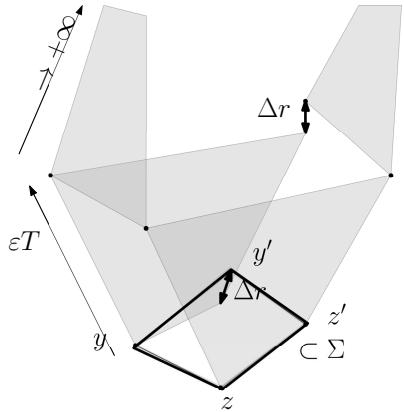}
\caption{Cocycles to get H\"older regularity for $B$}
\label{Fig-Bhold}
\end{center}
\end{figure}


\begin{eqnarray}
B(\FM(x)(x))-B(\FM(x'))&=& \int_{0}^{+\8}V(f_{t}(\FM(x)))-V(f_{t}(\gm(x)))dt\nonumber\\
&& \hskip 3cm  -\int_{0}^{+\8}V(f_{t}(\FM(x')))-V(f_{t}(\gm(x')))dt\nonumber\\
&=&  \int_{\eps.T}^{+\8}V(f_{t}(\FM(x)))-V(f_{t}(\gm(x)))dt\label{bound1}\\
&& \hskip 3cm -\int_{\eps T}^{+\8}V(f_{t}(\FM(x')))-V(f_{t}(\gm(x')))dt\label{bound2}\\
&&+\int_{0}^{\eps.T}V(f_{t}\circ f_{\rm(x)}(x))-V(f_{t}\circ f_{\rm(x)}(x'))dt\label{bound3}\\
&&-\int_{0}^{\eps.T}V(f_{t}\circ\gm(x))-V(f_{t}(\gm(x'))dt\label{bound4}\\
&& +\int_{-\DR}^{0}V(f_{t}(\FM)(x'))dt-\int_{\eps T}^{\eps T+\DR}V(f_{t}(\FM(x')))dt\label{bound5}.
\end{eqnarray}
The summand line \eqref{bound5} is easily bounded by $C.\DR$. 
Summands lines \eqref{bound1} and \eqref{bound2} are exponentially small in $T$ because of exponential contractions in the strong stable leaves. 
Summands in lines \eqref{bound3} and \eqref{bound4} are more difficult to deal with. For that we use that $V$ is $\al$-H\"older continuous. We set $\l'':=\log||Df_{1}||$. 

The distance between $f_{t}\circ f_{\rm(x)}(x)$ and $f_{t}\circ f_{\rm(x)}(x')$   increases exponentially fast in $t$. Therefore, for a fixed $\delta>0$, there is a time $T=T(x,x')$ such that 
$$d_{u}(f_{T}\circ f_{\rm(x)}(x),f_{T}\circ f_{\rm(x)}(x'))=\delta$$
holds, where $d_{u}$ means the distance along unstable leaves. On the other hand, expansion along unstable leaves is bounded by the norm of $Df$, and then, there exists a positive  number $\l'$ such that 
\begin{equation}
\label{eq-expanswu}
e^{\l'T}d_{u}(f_{\rm(x)}(x),f_{\rm(x)}(x'))\le \delta\le e^{\l''T}d_{u}(f_{\rm(x)}(x),f_{\rm(x)}(x')).
\end{equation}

Then, we pick $\eps$ such that $\eps\l''<\frac{\l'}2$. With these values  we get: 
\begin{eqnarray*}
|V(f_{t}\circ f_{\rm(x)}(x))-V(f_{t}\circ f_{\rm(x)}(x'))|&\le &C. \left(e^{\l''.t}d(f_{r}(x),f_{r}(x'))\right)^{\al}\\
\textrm{thus}&&\\
 \left|\int_{0}^{\eps.T}V(f_{t}\circ f_{\rm(x)}(x))-V(f_{t}\circ f_{\rm(x)}(x'))dt\right|&\le& \frac1{\al\l''}e^{\al\l''\eps T}d^{\al}(f_{r}(x),f_{r}(x'))\\
&\le&  \frac1{\al\l''}e^{\disp\frac{\disp \al\l''\eps}{\disp\l'} \l' T}d^{\al}(f_{r}(x),f_{r}(x'))\\
&\hskip -4cm\le&\hskip -2cm \frac1{\al\l''}\left(\frac\delta{d(f_{r}(x),f_{r}(x'))}\right)^{\frac{\disp \al\l''\eps}{\disp \l'} }d^{\al}(f_{r}(x),f_{r}(x')\\
&\le &\frac{\delta^{\frac{\disp \al\l''\eps}{\disp \l'}}}{\al\l''}\left(d(f_{r}(x),f_{r}(x'))\right)^{\disp\al(1-\frac{\disp \l''\eps}{ \l'})}.
\end{eqnarray*}
Now, remember the bi-H\"older relation between $\disp d(f_{r}(x),f_{r}(x'))$ and $d(\gm(x),\gm(x'))$. H\"older regularity for stable holonomy yields the same  kind of bound for the summand line \eqref{bound4}. 
\end{proof}

\subsection{From local to global Equilibrium State}
If $\wh W$ is a Borel function defined on $\CM$ we can study 
\begin{equation}
\label{equ-localequilstat}
\sup_{\mu\  \FM-inv}\left\{h_{\mu}(\FM)+\int \wh W\,d\mu\right\}.
\end{equation}
\begin{definition}
\label{def-localequil}
Any $\FM$-invariant probabilty which realizes the maximum in \eqref{equ-localequilstat} is called a local equilibrium state for $\wh W$. 
\end{definition}

If $\wh B$ is a Borel function and $\mu$ is $\FM$-invariant, if furthermore $\wh B$ belongs to $L^{1}(\mu)$, then 
$$\int \wh B\circ \FM\,d\mu=\int\wh B\,d\mu.$$

From this, we claim that it makes sense to study equilibrium state for the induced system $(\I, \gm)$ and a potentiel of the form $W-Z.\rm$, where $Z$ is a real parameter. Here, we present how we can deduce existence and uniqueness of a regular global equilibrium state from the existence of a local equilibrium state. Most of the ideas are from \cite{lep-survey}

\begin{theorem}
\label{theo-loctoglob}
Let $V$ be a $\al$-H\"older continuous potential. Set $\CP:=\max\left\{h_{\mu}(f_{1})+\int V\,d\mu\right\}$. Then, if $\mu$ is a $\gm$-invariant probability measure  satisfying:
\begin{enumerate}
\item $\disp \int\rm\,d\mu<+\8$,
\item $\mu$ satisfies $\disp h_{\mu}(\gm)+\int W-\CP. \rm\,d\mu=0$,
\end{enumerate}
then, there exists $\wh\mu$- which is $f$-invariant satisfying 
$$d\wh\mu\propto d\mu^{\CM}\otimes dt,$$
where $\mu^{\CM}$ is $\FM$-invariant and ${\TCT}_{*}\mu^{\CM}=\mu$, and moreover, $\wh\mu$ is an equilibrium state for $V$
\end{theorem}
\begin{proof}
The natural extension of $\mu$ can be seen as a $\FM$-invariant probability. This holds because $\CM$ is compact, $\FM$ expands in the ``horizontal'' $u$-direction (a consequence of the GALEO property) and contracts in the vertical $s$-direction. 

Let us denote this measure by $\mu^{\CM}$. By definition, ${\TCT}_{*}\mu^{\CM}=\mu$ holds. 
Then, $\disp \int\rm\,d\mu<+\8$ shows that we can find some $f$-invariant probability measure $\wh\mu$ satisfying 
$$d\wh\mu\propto d\mu^{\CM}\otimes dt.$$

Consequently, the flow is a suspended flow over $\CM$ with roof function $\rm$. Note that our first assumption yields that $\rm$ is in $L^{1}(\mu^{\CM})$. Moreover, $B$ is bounded, due to uniform contractions along the stable leaves. 
This shows that 
$$\int V_{\CM,\CP}\,d\mu^{\CM}=\int W-\CP\rm\,d\mu^{\CM}$$
holds. 

Our second assumption yields (using the Abramov Formula)
$$h_{\wh\mu}(f_{1})+\int V\,d\wh\mu-\CP=\wh\mu(\cup_{t\in[0,1]}f_{t}(\CM))\left(h_{\mu^{\CM}}(\FM)+\int \vmp\,d\mu^{\CM}\right)=0,$$
and then $\wh\mu$ is an equilibrium state for $V$. 
\end{proof}
\begin{remark}
\label{rem-pressneg}
The same computation based on Abramov formula shows that if $\wh\nu$ is $f_{1}$-invariant, if $\CM$ has positive* $\wh\nu$-measure and  if $\nu$ is the $\FM$-invariant probability such that 
$$\wh\nu\propto\nu\otimes dt$$
then, 
$$h_{\nu}(\FM)+\int W-\CP\rm\,d\nu\le 0,$$
with equality if and only if $\wh\nu$ is an equilibrium state (+ $\CM$ has positive* $\wh\nu$-measure). 
$\blacksquare$\end{remark}

\subsection{Local equilibrium states}

From now on, our goal is to study local equilibrium states for $(\CM,\FM)$ associated to potential  $W-\CP\rm$. For that purpose, we follow the method from \cite{lep-survey}  and introduced another parameter $Z$. 

Instead of studying local equilibrium states for $W-\rm\CP$ we will study local equilibrium states for $W-Z\rm$, where $Z$ is a real parameter. Actually, we will show that this can be done for any sufficiently large $Z$, say $Z> Z_{c}$. The main problem we will have to deal with, is that it is not a priori true that $Z_{c}<\CP$ holds. This will come as a consequence of the existence of a global regular equilibrium state. 

\subsubsection{Induced Transfer Operator}

\begin{definition}
For any $Z\in{\mathbb R}$, we define the linear operator in the following way: for any continuous function $\varphi$ on $\I$, for any $x\in\I$, we set
\begin{eqnarray*}
\Fi&\mapsto& \CL_{Z}(\Fi)\\
&&x\mapsto \CL_{Z}(\varphi)(x)=\sum_{y\in g_\CM^{-1}(x)} e^{W(y)-Zr_\CM(y)}\varphi(y).
\end{eqnarray*}
\end{definition}
For a fixed $\Fi$ and a fixed $x$, $\CL_{Z}(\Fi)(x)$ is not a true power series because return times are not necessarily integers. Nevertheless it behaves in spirit like a power series.
The first point is to make sure it is well defined. This is the purpose of next proposition. 

\begin{proposition}\label{Pro:convergence}
There exists $Z_{c}$ such that for every $Z>Z_{c}$, for every $\Fi$ and for every $x$, $\CL_{Z}(\Fi)(x)$ is well defines (converges) and for every $Z<Z_{c}$ and for every $x$, $\CL_{Z}(\BBone)(x)$ diverges. 
\end{proposition}
\begin{lemma}
\label{lem-Zcunif}
If $\CL_{Z}(x)$ converges for some $x\in\I$ then  $\CL_{Z}(x')$ converges for any $x'\in\I$. 
\end{lemma}
\begin{proof}[Proof of the Lemma]
By construction of $\CM$ (we refer here to the Markov property), every 1-cylinder of $\I$ contains exactly one preimage of any point in $\I$. If $x$ and $x'$ are in $\I$, we can associate by pair the preimages $y$ and $y'$ in the same1-cylinder. 

Then, Propositions \ref{prop-regularity} and \ref{prop-returndynlip} show that 
$$|W(y)-W(y')+Z(\rm(y)-\rm(y'))|\le K(1+|Z|)$$
for some universal constant $K$. 
This yields 
$$e^{-K(1+|Z|)}\le \frac{e^{W(y)-Z\rm(y)}}{e^{W(y')-Z\rm(y')}}\le e^{K(1+|Z|)},$$
and then $\CL_{Z}(\BBone)(x)\asymp e^{\pm K(1+|Z|)}\CL_{Z}(\BBone)(x')$. If one term converges, so does the other one. 
\end{proof}
\begin{proof}[Proof of Prop. \ref{Pro:convergence}]
Note that $\CL_{Z}$ is a positive operator and $\I$ is compact. Therefore, convergence for every $\Fi$ continuous is equivalent to the convergence for $\BBone$. Then, Lemma \ref{lem-Zcunif} shows that this last convergence does not depend on the choice of the reference point $x$. 

Therefore, we can choose any $x$ and check for convergence for 
$$\CL_{Z}(\BBone)(x):=\sum_{y,\ \gm(y)=x}e^{W(y)-Z\rm(y)}.$$
If we order the $\rm(y)$'s with respect to their integer part $\lfloor \rm(y)\rfloor$, the same argument as  for the proof of Lemma \ref{lem-Zcunif} shows that $\CL_{Z}(\BBone)(y)$ converges if and only if 
$$\sum_{n=1}^{+\8}\left(\sum_{\lfloor\rm(y)\rfloor=n}e^{W(y)}\right)e^{-nZ}$$
converges. This later sum is a power series and converges for $Z>Z_{c}$ and diverges for $Z<Z_{c}$  with 
\begin{equation}
\label{eq-Zc}
Z_{c}:=\limsup_{\ninf}\frac1n\log\left(\sum_{\lfloor\rm(y)\rfloor=n}e^{W(y)}\right).
\end{equation}
\end{proof}

\bigskip
Let $\ga$ be as in Prop. \ref{prop-regularity}. We recall that the $\ga$-H\"older norm is defined by 
$$||\Fi||_{\ga}=||\Fi||_{\8}+\sup_{x\neq y}\frac{|\Fi(x)-\Fi(y)|}{|x-y|^{\ga}}.$$

\subsubsection{Spectral decomposition for converging $\CL_{Z}$}

\begin{proposition}
\label{prop-doeblin}
If $\CL_{Z}(\BBone)(x)<+\8$ holds for some $x$, then, 
\begin{enumerate}
\item $\CL_{Z}$ acts on continuous functions. 

We denote by $\l_{Z}$ its spectral radius (on $\CC^{0}(\I)$) and set $\wt\CL_{Z}:=\disp\frac1{\l_{Z}}\CL_{Z}$. 
\item There exists $K=K(Z)$ such that for every $n$, for every $x$ and $y$ in $\I$, 
\begin{equation}
\label{equ-bowen}
e^{-K}\le \frac{\CL^{n}_{Z}(\BBone)(x)}{\CL_{Z}^{n}(\BBone)(y)} \le e^{K}. 
\end{equation}
If $Z$ stays in a compact, then $K(Z)$ can be chosen uniformly. 
\item $\wt\CL_{Z}$ acts on $\ga$-H\"older continuous functions $\CC^{\ga}(\I)$. 
\item $\wt\CL_{Z}$ satisfies the Doeblin-Fortet inequality on $\CC^{\ga}(\I)$:

there exist $0<a<1$ and $0<b$ such that 
$$\forall n,\ \forall \chi\in \CC^{\ga}(\I),\qquad ||\wt\CL_{Z}^{n+n_{0}}(\chi)||_{\ga}\le a||\chi||_{\ga}+b||\chi||_{\8}.$$
\end{enumerate}
\end{proposition}
\begin{proof}
These are standard computations involving $\CC^{\ga}(\I)$ and $\CC^{0}(\I)$. The key elements are Propositions \ref{prop-regularity} and \ref{prop-returndynlip}. Item $(2)$ is a direct consequence of these two propositions.  Proposition \ref{prop-returndynlip}  and compacteness for $\CT$ show that the quantity $K(Z)$ can be chosen uniformly if $Z$ stays in a compact set because $K(Z)$ is affine in $Z$ (variation is due to variation of $\rm^{n}$ in $n$-cylinders).

The last key point (to get $a<1$) is the uniform contraction in the horizontal direction for inverse branches of $\gm$. Actually, this follows from Observation \ref{obs-expan}. 
\end{proof}

We recall the Ionescu-Tulcea \& Marinescu theorem (see \cite{itm} and \cite{broise} for a proof adapted to Dynamical Systems).  We let the reader check that all assumptions hold in our case with $\CC^{\ga}(\I)\subset\CC^{0}(\I)$ and $\wt\CL_{Z}$ (with $\CL_{Z}(\BBone)(x)<+\8$ for some $x$). Items $(i)$ and $(v)$ follow from compactness of the unitary ball for $\CC^{\ga}(\I)$ into $\CC^{0}(\I)$. Items $(ii), (iii), (iv)$ follow from Prop. \ref{prop-doeblin}. 

\begin{theorem}[Ionescu-Tulcea \& Marinescu]\label{pajot1}
Let $(\CV,\Vert\;\Vert_\CV)$ and $(\CU,\Vert\;\Vert_\CU)$ be two ${\C}$-Banach 
spaces
such that $\CV\subset\CU$. We assume that 
\begin{itemize}
\item[(i)] if $(\varphi_n)_{n\in{\N}}$ is a sequence of functions in $\CV$
which converges in $(\CU,\Vert\;\Vert_\CU)$ to a function $\varphi$ and 
if for all $n\in{\N}$, $\Vert\varphi_n\Vert_{\CV}\leq C$ then \hbox{$\varphi
\in \CV$} and \hbox{$\Vert\varphi\Vert_\CV\leq C$},
\end{itemize}
and $\Phi$ an operator from $\CU$ to itself such that 
\begin{itemize}
\item[(ii)] $\Phi$ lets $\CV$ invariant and is bounded for
 $\Vert\;\Vert_\CV$ ;
\item[(iii)] $\sup_{n}\{ \Vert\Phi^{n}(\varphi)\Vert_\CU\;,\varphi\in\CV,\; 
\Vert\varphi\Vert_\CU\leq 1\}<+\8$ ;
\item[(iv)] there exists an integer $n_0$ and two constants $0<a<1$ and
$0\leq b<+\8$ such that for all $\varphi\in\CV$ we have 
\hbox{$\Vert\Phi^{n_0}(\varphi)\Vert_\CV\leq a\Vert\varphi\Vert_\CV+b
\Vert\varphi\Vert_\CU$ ;}
\item[(v)] if $\CX$ is a bounded subset of $(\CV,\Vert\;\Vert_\CV)$ then 
$\Phi^{n_0}(\CX)$ has compact closure in  
 $(\CU,\Vert\;\Vert_\CU)$.
\end{itemize}
Then $\Phi$ has a finite number of eigenvalues of norm 1 :
$\lambda_1\ldots\lambda_p$ and $\Phi$ can be written  
\hbox{$\Phi=\sum_{i=1}^p\lambda_i\Phi_i\;+\Psi$, } where the $\Phi_i$ are 
linear bounded operators from $\CV$ to $\Phi(\CV)$ of finite
dimension image contained in
 $\CV$, and where $\Psi$ is a linear bounded operator with spectral radius
 $\rho(\Psi)<1$ in $(\CV,\Vert\;\Vert_\CV)$.

\noindent
Moreover the following holds :
$\Phi_i.\Phi_j=\Phi_j.\Phi_i=0$ for all $i\not = j $, 
$\Phi_i.\Phi_i=\Phi_i$  for all $i$, and $\Phi_i.\Psi=\Psi.\Phi_i=0$
for all $i$.
\end{theorem}

\subsubsection{Finer spectral decomposition and consequences}

\begin{proposition}
\label{prop-spectre}
With previous assumptions and notations, 
$\lambda_{Z}$ is a simple single dominating eigenvalue.
\end{proposition}
\begin{proof}
We use the cone-theory for operators as it is studied in \cite[chap. 1\& 2]{Kras64}.  
We claim that the set $\CK$ of non-negative $\ga$-H\"older continuous functions is a \emph{solid and reproductible cone} . Solid means it as non-empty interior and reproductible means $$\CC^{\ga}(\I)=\CK-\CK.$$

$\bullet$ {\bf Step one}. We prove that for any $\Fi\not\equiv 0\in \CK$, there exists $q$ such that $\CL_{\phi}^{p}(\Fi)$ belongs to $\inte K$. 

If $\Fi\neq 0$, there exists some $q$-cylinder, say $C_{q}$ such that for every $y$ in $C_{q}$, $\Fi(y)>0$. Then, for every $x$ in $\I$, there exists $y\in C_{q}$ such that $\gm^{q}(y)=x$, and hence 
$$\CL_{Z}^{q}(\Fi)(x)\ge e^{S_{q}(W)(y)-Z.\rm^{q}(y)}\Fi(y)>0.$$

$\bullet$ {\bf Step two}. End of the proof. 
We deduce from step one that $\CL_{Z}$ is strongly positive (see \cite[p.60]{Kras64}). Therefore it is $u$-positive for any $u\in \inte \CK$. From \cite[Th. 2.10, 2.11 and 2.13]{Kras64}  we deduce that $\l_{Z}$ is a simple single dominating eigenvalue.
\end{proof}

Re-employing notations from above, we get $p=1$. As $\CL_{Z}$ acts on continuous function defined on the compact set $\I$, its dual operator acts on the measures. 

The Schauder-Tychonoff theorem yields that there exists $\nu_{Z}$ a probability measure on $\I$ such that 
$$\CL_{Z}^{*}\nu_{Z}=\wt\l_{Z}\nu_{Z},$$
with $\wt\l_{Z}=\int \CL_{Z}(\BBone)\,d\nu_{Z}$. Then, Item $(2)$ of Proposition \ref{prop-doeblin} shows that $\wt\l_{Z}=\l_{Z}$. 

Moreover, we emphasize the double inequality for every $x$ and every $n$, 
\begin{equation}
\label{eq-controldistolz}
\l_{Z}^{n}e^{-K}\le \CL_{Z}(\BBone)^{n}(x)\le \l_{Z}^{n}e^{K}, 
\end{equation} 
where $K$ is the same constant as in Proposition \ref{prop-doeblin} item 2.

Then, Proposition \ref{prop-spectre} yields the existence of  $H_{Z}$ a positive $\ga$-H\"older function such that 
$$\CL_{Z}(H_{Z})=\l_{Z}H_{Z}.$$
The function $H_{Z}$ is unique if we add the condition $\disp \int H_{Z}\,d\nu_{Z}=1$. $H_{Z}$ and $\nu_{Z}$ are respectively referred to as the eigen-function and the eigen-measure. 

The spectral decomposition for $\CL_{Z}$ means 
\begin{equation}\label{eq-lczspec}
\forall\Fi\in\CC^{\ga}(\I),\ \forall n,\ \CL_{Z}^{n}(\Fi)=\l_{Z}^{n}\int\Fi\,d\nu_{Z}. H_{Z}+\l_{Z}^{n}\Psi^{n}(\Fi).
\end{equation}

The measure $\mu_{Z}$ defined by 
$$d\mu_{Z}=H_{Z}d\nu_{Z},$$
is a Gibbs measure, in the sense that for every $n$ cylinder $C_{n}$ and for every $x\in C_{n}$, 
$$\mu_{Z}(C_{n})\asymp\footnotemark \ e^{S_{n}(W)(x)-\rm^{n}(x)Z}e^{\pm K},$$
\footnotetext{Note that the Birkhoff sum is done with respect to $\gm$.}
for some universal constant $K$.  It is a standard computation that, 

on the one hand 
$$h_{\mu_{Z}}(\gm)+\int W-Z\rm\,d\mu_{Z}=\log\l_{Z},$$
and on the other hand,  for any other $\gm$-invariant probability $\mu$, 
 $$h_{\mu}(\gm)+\int W-Z\rm\,d\mu<\log\l_{Z}.$$
In other words we have proved: 
\begin{theorem}
\label{th-localequil1d}
If $\CL_{Z}(\BBone)(x)<+\8$ for some $x\in \I$, then $\mu_{Z}$ is the unique equilibrium state for $(\I,\gm)$ associated to $W-Z.\rm$.
\end{theorem}
Furthermore, if $\mu_{Z}^{\CM}$ denotes the natural extension of $\mu_{Z}$ (see. the Proof of Th. \ref{theo-loctoglob}) then it has full support in $\CM$. This holds because any vertical band over a cylinder (of any generation) is sent to and horizontal strip, and any vertical band has positive measure. 

We let the reader check the following result. The diffeo-version  for this result can  be found in \cite{lep-survey} and which comes from the Abramov formula:
\begin{theorem}
\label{th-muzextended}
If $Z>Z_{c}$, then $\disp\int \rm\,d\mu_{Z}<+\8$ and there exists a unique $f_{1}$-invariant probability measure $\wh\mu_{Z}$ such that 
\begin{enumerate}
\item $\disp d\wh\mu_{Z}\asymp d\mu^{\CM}\otimes dt$ with $\mu_{Z}^{\CM}$ a $\FM$-invariant probability measure. 
\item ${\TCT}_{*}\mu_{Z}^{\CM}=\mu_{Z}$. 
\item $\disp h_{\wh\mu_{Z}}(f_{1})+\int V\,d\wh\mu_{Z}=Z+\frac1{\int\rm\,d\mu_{Z}}\log\l_{Z}$. 
\end{enumerate}

\end{theorem}

\subsection{Upper bound for $Z_{c}$}
In the previous subsection we have seen that $W-Z.\rm$ admits a local equilibrium state as soon as $\CL_{Z}(\BBone)$ is well defined, that is, the series $\CL_{Z}(\BBone)(z)$  does converge for any (or at least one) $z\in\I$. This holds if $Z>Z_{c}$, by definition of $Z_{c}$, but may also hold for $Z=Z_{c}$. Here, we prove that $Z_{c}\le \CP$. More precisely, we prove that $Z_{c}\le \CP$ always holds and $Z_{c}<\CP$ holds if $V$ admits a regular equilibrium state. 

\begin{proposition}
\label{prop-controlZc}
There exists a $f_{t}$-invariant measure $\wh\mu$ such that $$Z_{c}\le h_{\wh\mu}(f_{1})+\int V\,d\wh\mu$$ holds. Moreover, $\wh\mu(\cup_{t[0,1]}f_{t}(\inte \CM))=0$. 
\end{proposition}
\begin{proof}
We pick some $\xi$ in $\I$. We recall $Z_{c}$ that satisfies  (see \eqref{eq-Zc})
$$Z_{c}=\limsup_{\ninf}\frac1n\log\left(\sum_{\lfloor\rm(y)\rfloor=n}e^{W(y)}\right),$$
with $\gm(y)=\xi$. 
Each 1-cylinder is mapped by $\gm$ onto $\I$, thus contains a unique\footnote{due to expansion} $\gm$-fixed point. If $C_{i,n}$ is a 1-cylinder satisfying 
\begin{enumerate}
\item $y\in C_{i,n}$, 
\item $\lfloor \rm(y)\rfloor=n$,
\end{enumerate}
we denote by $\xi_{i,n}$ the fixed point in $C_{i,n}$. The contraction in the stable leaf $W^{ss}(\xi_{i,n})$ yields that $\xi_{i,n}=:\TCT(\wh\xi_{i,n})$ with $\wh\xi_{i,n}$ in $\CM$ and $\FM(\wh\xi_{i,n})=\wh\xi_{i,n}$. 

Let $r_{i,n}$ be the period for $\wh\xi_{i,n}$, that is 
$$f_{r_{i,n}}(\wh\xi_{i,n})=\wh\xi_{i,n}.$$
Let us set $A_{i,n}:=\disp\int_{0}^{r_{i,n}} V(f_{t}(\wh\xi_{i,n}))\,dt=\vmo(\wh\xi_{i,n})$. 
Let $\wh C_{i,n}$ be the intersection of the vertical band $\TCT^{-1}(C_{i,n})$ and the horizontal strip $\FM(\TCT^{-1}(C_{i,n}))$. 
Note that all the $C_{i,n}$ are disjoints. This holds because two cylinders have empty interior intersection and the boundaries are preimages of the periodic orbit which  contains $p_{\pm}$. Moreover, $p_{-}$ and $p_{+}$ are two consecutive points (for the order relation in the interval), and then no other point of the periodic orbit lies  between them. Now, if two 1-cylinders do intersect on their boundary, this would produce a point of the periodic orbit between $p_{-}$ and $p_{+}$. 

From this we deduce that all the $\wh C_{i,n}$ are finitely many disjoint sets. We denote this collection by $\wh\CC_{n}$.
Then, we construct a $\FM$-invariant measure $\mu_{n}$ in the following way: 
\begin{enumerate}
\item $\mu_{n}(\wh C_{i,n})=\frac{\disp e^{A_{i,n}}}{\disp\sum_{j}e^{A_{j,n}}}=:p_{i}.$
\item $\mu_{n}(\wh C_{i_{0},n}\cap \FM^{-1}(\wh C_{i_{1},n})\cap\ldots\cap \FM^{-k}(\wh C_{i_{k},n}))=\prod_{j=0}^{k}p_{i_{j}}$. 
\end{enumerate}
The measure can be extended to a $f_{t}$-invariant measure $\wh\mu_{n}$ because $\rm\asymp n\pm K$ (see below).

From  Proposition \ref{prop-regularity}
 we get that there exists some universal constant $K$ such that 
$$e^{-K}\le\frac{e^{W(y)}}{e^{W(\xi_{i,n})}}\le e^{K},$$
where $y$ belongs to $C_{i,n}$ and $\gm(y)=\xi$. 
This yields 
\begin{equation}
\label{eq-approx1Zc}
\sum_{\substack{\gm(y)=\xi,\\ \lfloor \rm(y)\rfloor=n }}e^{W(y)}=\sum_{j}e^{W(\wh\xi_{j,n})}e^{\pm K}.
\end{equation}
As $B$ is bounded, equality \eqref{equ-coboundaryAW} shows that 
$\disp A_{i,n}=W(\wh\xi_{j,n})\pm2K$ holds. Then, equality \eqref{eq-approx1Zc} yields 
\begin{equation}
\label{eq-approx2Zc}
\sum_{\substack{\gm(y)=\xi,\\ \lfloor \rm(y)\rfloor=n }}e^{W(y)}=\sum_{j}e^{Ai_{j,n})}e^{\pm 3K}.
\end{equation}
On the other hand, 
\begin{eqnarray*}
\frac1N H_{\mu_{n}}(\bigvee_{0}^{N-1}\FM^{-k}(\wh\CC_{n}))&=&-\frac1N\sum_{\vec i=(i_{0},\ldots i_{N-1})}\prod_{k=0}^{N-1}p_{i_{k}}\log\prod_{k=0}^{N-1}p_{i_{k}}\\
&=& -\frac1N\sum_{\vec i=(i_{0},\ldots i_{N-1})}\sum_{k=0}^{N-1}\prod_{k=0}^{N-1}p_{i_{j}}\log p_{i_{k}}\\
&=& -\frac1{N}\sum_{k=0}^{N-1}\sum_{j} p_{j}\log p_{j}\sum_{\substack{i=(i_{0},\ldots i_{N-1})\\ i_{k}=j\\}}\prod_{\substack{l=0\\l\neq k}}^{N-1}p_{i_{l}}\\
&=& -\sum_{j}p_{j}\log p_{j}. 
\end{eqnarray*}
Hence, 
\begin{equation}
\label{equ1-entropymun}
h_{\mu_{n}}(\FM)=-\sum_{j}p_{j}\log p_{j}.
\end{equation}
Furthermore  
\begin{equation}
\label{eq-inteW}
\int \vmo \,d\mu_{n}=\sum_{j}p_{j}A_{j,n}\pm 3K.
\end{equation}
Equalities \eqref{equ1-entropymun} and \eqref{eq-inteW} yield 
$$h_{\mu_{n}}(\FM)+\int\vmo\,d\mu_{n}\ge \sum_{j}-p_{j}\log p_{j}+p_{j}A_{j,n}-3K.$$
Now, it is well known that $\sum_{j}-q_{j}\log q_{j}+q_{j}A_{j,n}$ with condition $\sum q_{j}=1$ and the $q_{j}$'s are non negative is maximal for $q_{j}=p_{j}$ with value $\log\left(\sum e^{A_{j,n}}\right)$ (see first pages of \cite{Bowenlnm}).

And finally, we recall that Prop. \ref{prop-returndynlip} 
 shows that for every $z\in \wh C_{i,n}$, $\rm(z)=n\pm K$. Then, the Abramov formula and \eqref{eq-approx2Zc} yield
 $$
\frac1n\log\left(\sum_{\substack{\gm(y)=\xi,\\ \lfloor \rm(y)\rfloor=n }}e^{W(y)}\right)\le \frac{n+K}{n\int \rm\,d\mu_{n}}\left(h_{\mu_{n}}(\FM)+\int\vmo\,d\mu_{n}\right)+\frac{6K}{n}.
$$
If $\wh\mu_{n}$ denotes the $f_{t}$-invariant probability obtained from $\mu_{n}$, then we have   

\begin{equation}
\label{eq-estiZc}
\frac1n\log\left(\sum_{\substack{\gm(y)=\xi,\\ \lfloor \rm(y)\rfloor=n }}e^{W(y)}\right)\le \frac{n+K}{n}\left(h_{\wh\mu_{n}}(f_{1})+\int V\,d\wh\mu_{n}\right)+\frac{6K}{n}.
\end{equation}
Now, let us consider any subsequence $n_{k}\to+\8$ such that the left hand side term in \eqref{eq-estiZc} goes to $Z_{c}$, and then  a smaller subsequence such that $\wh\mu_{n_{k}}$ converges to some $f_{t}$-invariant probability $\wh\mu$. The upper semi-continuity for entropy yields 
$$Z_{c}\le h_{\wh\mu}(f_{1})+\int V\,d\wh\mu.$$

It remains to estimate $\wh\mu(\CM)$. Note that $\rm(z)=n\pm K$ yields $\int\rm\, d\mu_{n}=n\pm K$, and then 
$$\wh\mu_{n}(\cup_{0\le t\le 1}f_{t}(\CM))\approx\frac1n.$$
A standard computation shows that at the limit, the interior  of $\cup_{0\le t\le 1}f_{t}(\CM)$ cannot get positive mesure. 
\end{proof}

We deduce a very important result from the estimation of $Z_{c}$:
\begin{corollary}
\label{coro-lpcv}
Inequality $Z_{c}\le \CP$ holds and $\CL_{\CP}(\BBone)(x)<+\8$  holds for any $x$ in $\I$. Moreover, $\l_{\CP}\le 1$. 
\end{corollary}
\begin{proof}
We copy here method (and results) from \cite{lep-survey}. 

By Proposition \ref{prop-controlZc}, 
$Z_{c}\le h_{\wh\mu}(f_{1})+\int V\,d\wh\mu$ holds and $h_{\wh\mu}(f_{1})+\int V\,d\wh\mu\le\CP$ holds by definition of $\CP$. 

Therefore, for any $Z>\CP$, $Z>Z_{c}$ holds and we can apply Theorem \ref{th-muzextended}. In particular the 3nd point states 
$$\disp h_{\wh\mu_{Z}}(f_{1})+\int V\,d\wh\mu_{Z}=Z+\frac1{\int\rm\,d\mu_{Z}}\log\l_{Z},$$
with $Z>\CP$, which yields $\log\l_Z<0$. 


Furthermore, $Z\mapsto \CL_{Z}(\BBone)(x)$ is decreasing. If $Z\downarrow \CP$, Inequality \eqref{eq-controldistolz} shows that 
$$\disp\CL_{\CP}(\BBone)(x)\le e^{K}$$
which means that there is convergence of the induced operator for $Z=\CP$. 
The same argument also shows that for every $n$,  
$$\disp\CL_{\CP}^{n}(\BBone)(x)\le e^{K}$$
holds which implies that $\l_{\CP}\le 1$. 
\end{proof}

\section{End of the Proofs of Theorems}\label{sec-endproof}
\subsection{Proof of first part of Theorem \ref{th-main1}}
If no regular measure for $V$ in an equilibrium state, then the theorem holds. 
Let us thus assume that $V$ admits two different regular equilibrium states. We will  produce some contradiction. 

We denote these equilibrium states by $\wh\mu_{1,V}$ and $\wh\mu_{2,V}$. We consider $x_{i}$, $i=1,2$ a regular point with respect to $\wh\mu_{i,V}$ and we assume it satisfies any condition which holds $\wh\mu_{i,V}$-\AE. Moreover, we assume that $x_{i}$ is a density point for all the (finitely many) properties  we are going to invoke below.

We can construct two mille-feuilles, say $\CM_{1}$ and $\CM_{2}$, each one with positive* $\wh\mu_{i,V}$ measure. 

\begin{proposition}
\label{prop-genemillefeuilbis}
There exists a generalized mille-feuilles (see Def. \ref{def-generalizedmilfeui}) with positive* $\wh\mu_{1,V}$ and  $\wh\mu_{2,V}$ measure. 
\end{proposition}
\begin{proof}
First, let us fix notations. For $\CM_{i}$, the measure $\wh\mu_{i,V}$ can be written under the form 
$$\mu_{i,V}^{\CM_{i}}\otimes dt,$$
where $\mu_{i,V}^{\CM_{i}}$ is $F_{\CM_{i}}$-invariant (reemploying previous notations). 
The basis is $[p_{i,-},p_{i,+}]$ and projection of $\mu_{i,V}^{\CM_{i}}$ on the basis is 
a $g_{\CM_{i}}$-invariant measure $\mu_{i,V}$.

Note that the assumption ``$\CM_{i}$ has positive* $\wh\mu_{i,V}$-measure'' means that some box $\disp\bigcup_{t\in[-\eps,\eps]}f_{t}(\CM_{i})$ has positive measure, thus the expectation of the return time in the box is finite. In the box the measure is of the form $\mu_{i,V}^{\CM_{i}}\otimes dt$ and returns times are almost constant along the flow direction locally and  in the box. This yields that the return time has finite expectation for $\mu_{i,V}^{\CM_{i}}$. 

Furthermore, by construction of the mille-feuilles, return times are constant along $W^{ss}_{\CM_{i}}$-fibers. If $\mu_{i,V}$ denotes the projection onto the basis of $\mu_{i,V}^{\CM_{i}}$, it is a $g_{\CM_{i}}$-invariant measure  and 
$$\int r_{\CM_{i}}\,d\mu_{i,V}=\int r_{\CM_{i}}\,d\mu_{i,V}^{\CM_{i}}.$$
By the Abramov formula, 
$$h_{\mu_{i,V}}(g_{\CM_{i}})+\int W-\CP r_{\CM_{i}}\,d\mu_{i,V}=h_{\mu_{i,V}}(g_{\CM_{i}})+\int V_{\CM_{i},\CP}\,d\mu_{i,V}=0,$$
and then Corollary \ref{coro-lpcv} and Theorem \ref{th-localequil1d} show that $\mu_{i,V}$ is the measure $\mu_{\CP,i}$, obtained on $\CM_{i}$ following the work done above with $Z=\CP$. 

The same work can be done on any generalized mille-feuilles constructed over a $n$-cylinder of each $\CM_{i}$. 
Let us then consider $\CM_{1}$. We claim that there exists some periodic point in $\S_{1}$ with projection in $[p_{-,1},p_{+,1}]$ whose orbit does belong to $\S_{2}$ and with projection in $[p_{-,2},p_{+,2}]$. This holds because $\mu_{SRB}$ has full support and we can create a periodic orbit close to any $\mu_{SRB}$-regular orbit for any fixed interval if time $t\in [0,T]$. 

To fix notation we denote by $P$ the point in $\S_{1}$, $Q=f_{\tau}(P)$ the point in $\S_{2}$. By construction $P$ belongs to the vertical band over some $n$-cylinder $C_{n,1}$ and $Q$ belongs to the vertical band over some $m$-cylinder $C_{m,2}$.

This orbit is hyperbolic. We can thus always increase the constants $\kappa_{1,1}$, $\kappa_{2,1}$ and  $\kappa_{1,2}$, $\kappa_{2,2}$ and $n$ and $m$  such that the unstable leaves at $P$ and $Q$ become eligible curves of type 2 for the generalized mille-feuilles over $C_{n,1}$ and $C_{m,2}$. 

By uniqueness of the local equilibrium state, the local equilibrium state for the 1d dynamics associated to these 2 generalized mille-feuilles are the restriction of $\mu_{i,V}$ (with renormalization to get probabilities). 

The generalized mille-feuille over $C_{n,1}$ has positive* $\wh\mu_{1,V}$-measure ad as the local equilibrium state is a Gibbs measure, we can find a very small box around $P$ in $\S_{1}$ with positive* $\wh\mu_{1,V}$-measure. We can then push it forward by the flow and this yields that the generalized mille-feuilles over $C_{m,2}$ has positive* $\wh\mu_{1,V}$-measure. 

This finishes the proof. 
\end{proof}

To conclude the proof of Theorem \ref{th-main1} we just have to say that the restriction and projection of $\wh\mu_{i,V}$ on the basis of the generalized mille-feuilles from Prop. \ref{prop-genemillefeuilbis}, say $\CM, $do coincide because they are the unique local equilibrium state for $V-\CP\rm$. Therefore 
$$\wh\mu_{1,V}=\wh\mu_{2,V},$$
and the Theorem is proved.

\subsection{Proof of Theorem \ref{th-main2}}
If $V$ is a potential and $\s$ is a singularity for $f$ such that $\delta_{\s}$ is an equilibrium state for $V$ then we have for any invariant probability $\mu$, 
$$h_{\mu}(f_{1})+\int V\,d\mu\le h_{\delta_{\s}}(f_{1})+\int V\,d\delta_{\s}=V(\s).$$
This yields 
$$\forall \mu,\ \int V\,d\mu\le V(\s),$$
which means that $\delta_{\s}$ is a $V$-maximizing measure. The same argument works for any $\be.V$ with $\be> 0$.

Consequently, a singularity can be the support of an equilibrium state for $\be.V$ with $\be>0$ only if the Dirac measure at the singularity is also a $V$-maximizing measure.

Theorem \ref{th-main2} is then just the contrapositive of this observation. If no singularity is a $V$-maximizing measure, no singularity can be an equilibrium state for $\be.V$ with $\be>0$. As there exists at least one, it must me a regular measure, thus it is unique. 

Note that analyticity for $\CP(\be)$ (and $\be\ge 0$) will follow from Theorem \ref{th-main1} ``case 1''.

\subsection{Proof of Theorem \ref{th-main3}}
Theorem \ref{th-main3} follows from the same observation. No singularity can be a measure of maximal entropy because $h_{top}>0$ and singularities have zero entropy. As the measure of maximal entropy does exist, it is regular thus unique.

\subsection{End of the proof of Theorem \ref{th-main1}}
We go back to the general case for $V$. First, we claim that there exists $\eps>0$ such that for any $\be\in[0,\eps]$, there exists a unique equilibrium state for $\be.V$ and it is a regular measure. This holds because any singularity has zero entropy and then, for $\be$ sufficiently small, 
$$h_{\mu_{top}}(f_{1})+\be\int Vd\mu_{top} >\be\max_{\s\ singularity} V(\s),$$
where $\mu_{top}$ is the measure of maximal entropy.
Therefore no singularity can be an equilibrium state for these small (but positive) $\be$. 

Now, if there exists $\be_{c}$ such that some singularity is an equilibrium state for $\be_{c}.V$, then it is also a $V$-maximizing measure. The pressure function is convex and  has for asymptote at $+\8$ the line 
$$h+\be.A(v),$$
where $A(v)$ stands for the maximal value for $\int V\,d\mu$ (with $\mu$) and $h$ is the maximal entropy among $V$-maximizing measures. 

Therefore, if $\be_{c}$ exists as above, as singularities have zero entropy, this yields that $h=0$. 
Consequently, $\CP(\be)$ touches its asymptote at $\be_{c}$ and for convexity reason, $\CP(\be)\equiv \be. A(v)$ for any $\be>\be_{c}$. 

Now, for $0<\be<\be_{c}$, the first part of Theorem \ref{th-main1} says that there exists a unique equilibrium state  and it has full support. We remind Prop. \ref{prop-controlZc}, which states
$$Z_{c}(\be)\le h_{\wh\mu^{\be}}(f_{1})+\be.\int V\,d\wh\mu^{\be},$$
with $\wh\mu^{\be}(\inte{\CM})=0$. This last condition shows that $\wh\mu^{\be}$ cannot be the equilibrium state associated to $\be. V$ because the equilibrium state has full support. 
Thus we get 
$$Z_{c}(\be)<\CP(\be).$$
This yields the implicit formula: 
\begin{equation}
\label{equ-implicipbeta}
1=\l_{\CP(\be)}.
\end{equation}
Now, we claim that 
\begin{equation}
\label{equ2-implicibeta}
\disp {\frac{\partial \l_{Z,\be}}{\partial \be}}_{|Z=\CP(\be)}\not=0,
\end{equation} 
where $\l_{Z,\be}$ is the spectral radius for $\CL_{Z}$ and $\be.V$. 
Assuming this claim holds, we show now how we can finish the proof of Theorem \ref{th-main1}. 

We let the reader check that for fixed $\be$, $Z\mapsto \l_{Z,\be}$ is $\C$-analytic in the complex domain $|Z|>Z_{c}(\be)$. This holds because $\l_{Z,\be}$ is a simple dominating eigenvalue, thus this local analyticity  and connectedness shows that it is globally analytic. 

Furthermore, for fixed $Z$, $\be\mapsto \l_{Z,\be}$ is analytic in some complex neighborhood of  $|\l_{Z,\be}|\le 1$. Then,  the implicit function theorem for analytic functions (see \cite{Range}) yields that $\CP(\be)$ is locally analytic and then globally analytic as long as $\be<\be_{c}$, again by connectedness.

To finish the proof, we thus just need to prove \eqref{equ2-implicibeta}. Actually this holds because for fixed $\be<\be_{c}$ and $Z=\CP(\be)>Z_{c}(\be)$, if we use notations from Theorem \ref{th-muzextended} with $\be.V$ and $Z>Z_{c}(\be)$, 
 Equality \eqref{eq-lczspec} yields 
 $$\log\l_{Z,\be}=\lim_{\ninf}\frac1n\log\CL_{Z}^{n}(\BBone)(x),$$
 for any $x$ in $\I$. We let the reader check that normal convergence and $\CP(\be)>Z_{c}(\be)$ yield
 $$\frac{\partial\l_{Z,\be}}{\partial Z}=\l_{Z,\be}\lim_{\ninf}\frac1n\frac{\partial\CL_{Z}^{n}(\BBone)(x)}{\partial Z}\frac1{\CL_{Z}^{n}(\BBone)(x)}.$$
 Furthermore, $\disp \frac{\partial\CL_{Z}^{n}(\BBone)(x)}{\partial Z}=-\CL_{Z}^{n}(\rm^{n})(x)$ and 
  the Lebesgue Dominated convergence theorem yields 
 $$\lim_{\ninf}\frac1n\CL_{Z}^{n}(\rm^{n})(x)\frac1{\CL_{Z}^{n}(\BBone)(x)}=\int \rm\,d\nu_{Z}.$$
 If  $\wh\mu_{\be}$ denotes the unique equilibrium state for $\be.V$, we get
$$0\neq\frac1{\wh\mu_{\CP(\be)}(\cup_{t\in[0,1]}f_{t}(\CM))}=\int \rm\,d\mu_{\CP(\be)}=e^{\pm K}\int \rm\,d\nu_{\CP(\be)},$$
and this shows that $\disp{\frac{\partial\l_{Z,\be}}{\partial Z}}_{|Z=\CP(\be)}$ is not zero.

\bibliographystyle{plain}
\bibliography{bibliodawei2}

R. Leplaideur\\
ISEA, Université de la Nouvelle Cal\'edonie.

\smallskip
LMBA, UMR6205\\
Universit\'e de Brest\\

\texttt{Renaud.Leplaideur@unc.nc}\\
\texttt{http://rleplaideur.perso.math.cnrs.fr/}

\end{document}